\documentclass[11pt]{article}

\usepackage[margin=1in]{geometry}
\usepackage{amsthm}
\usepackage{amssymb}
\usepackage{amsmath}
\usepackage{mathrsfs}
\usepackage{graphicx}
\usepackage{url}
\usepackage{color}
\usepackage[table]{xcolor}
\usepackage{tikz}

\definecolor{or}{HTML}{F58634}

\newtheorem{theorem}{Theorem}[section]
\newtheorem{corollary}[theorem]{Corollary}
\newtheorem{lemma}[theorem]{Lemma}
\newtheorem{proposition}[theorem]{Proposition}

\theoremstyle{definition}

\newtheorem{example}[theorem]{Example}
\newtheorem{definition}[theorem]{Definition}

\title{Doubly transitive lines I: Higman pairs and roux}

\author{Joseph W.\ Iverson\footnote{Department of Mathematics, Iowa State University, Ames, IA}\qquad Dustin~G.~Mixon\footnote{Department of Mathematics, The Ohio State University, Columbus, OH} \footnote{Translational Data Analytics Institute, The Ohio State University, Columbus, OH}}
\date{}

\begin{document}
\maketitle

\begin{abstract}
We study lines through the origin of finite-dimensional complex vector spaces that enjoy a doubly transitive automorphism group.
In doing so, we make fundamental connections with both discrete geometry and algebraic combinatorics.
In particular, we show that doubly transitive lines are necessarily optimal packings in complex projective space, and we introduce a fruitful generalization of regular abelian distance-regular antipodal covers of the complete graph.
\end{abstract}

\section{Introduction}
\label{sec.intro}

Given a sequence $\mathscr{L}$ of lines through the origin of $\mathbb{C}^d$, we consider all unitary operators that permute these lines, and we refer to such permutations as automorphisms of $\mathscr{L}$.
We are interested in \textbf{doubly transitive lines}, that is, lines that enjoy a doubly transitive automorphism group.
(Recall that a permutation group $G\leq\operatorname{Sym}(X)$ is doubly transitive if for every $x_1,x_2,y_1,y_2\in X$ with $x_1\neq x_2$ and $y_1\neq y_2$, there exists $g\in G$ such that $g\cdot x_1=y_1$ and $g\cdot x_2=y_2$.)
This paper is the first in a series that studies doubly transitive lines.
Our interest is driven by a surprising connection with a fundamental problem in discrete geometry.
Consider the task of packing lines through the origin so that the minimum distance between any two is as large as possible.
Given unit-norm representatives $\varphi_i\in\ell_i$ for $i\in[n]:=\{1,\ldots,n\}$ of lines $\mathscr{L}=\{\ell_i\}_{i\in[n]}$, then the \textbf{coherence} of the sequence $\Phi=\{\varphi_i\}_{i\in[n]}$ is defined by
\[
\mu(\Phi):=
\max_{\substack{i,j\in[n]\\ i\neq j}}|\langle \varphi_i,\varphi_j\rangle|.
\] 
Minimizing the coherence of $\Phi$ corresponds to maximizing the minimum pairwise \emph{chordal distance} of $\mathscr{L}$~\cite{ConwayHS:96}.
Sequences of unit vectors that minimize coherence find applications in compressed sensing~\cite{BandeiraFMW:13}, multiple description coding~\cite{StrohmerH:03}, digital fingerprinting~\cite{MixonQKF:13}, and quantum state tomography~\cite{RenesBSC:04}.
As we will soon see, unit-norm representatives of $n>d$ doubly transitive lines that span $\mathbb{C}^d$ necessarily minimize coherence, and these objects enjoy a fruitful combinatorial generalization.

The remainder of this introduction provides an overview of our approach.
In Subsection~\ref{subsec.preliminaries}, we give preliminaries from both frame theory and association schemes.
Next, Subsection~\ref{subsec.example main results} walks through an explicit motivating example that illustrates the primary objects of interest in this paper before enunciating our main results.
Subsection~\ref{subsec.context} then elaborates on various precursors to these objects and their relationships.
Finally, Subsection~\ref{subsec.outline} outlines the remainder of the paper.

\subsection{Preliminaries}
\label{subsec.preliminaries}

\subsubsection{Frames and equiangular lines}

First, we review the basics of frames and equiangular lines; see~\cite{Waldron:18} for a complete treatment.
In this paper, we are principally concerned with lines that have unit-norm representatives that minimize coherence.
One popular lower bound on the coherence is the \textbf{Welch bound}~\cite{Welch:74}, given by
\begin{equation}
\label{eq.welch bound}
\mu(\Phi)
\geq\sqrt{\frac{n-d}{d(n-1)}}.
\end{equation}
Equality is achieved in the Welch bound precisely when the sequence of vectors form an \textbf{equiangular tight frame (ETF)}~\cite{StrohmerH:03}, meaning there exist $\alpha,\beta\in \mathbb{R}$ such that
\[
\sum_{i\in[n]}\varphi_i\varphi_i^*=\alpha I,
\qquad
\qquad
|\langle \varphi_i,\varphi_j\rangle|^2=\beta
\qquad \forall i,j\in[n],~i\neq j.
\]
Equivalently, the \textbf{Gram matrix} $(\langle \varphi_j,\varphi_i\rangle)_{ij}$ is a scalar multiple of an orthogonal projection matrix whose off-diagonal entries all have the same modulus.

As a parallel pursuit, there is a large literature on linearly dependent equiangular lines.
(Here, the \textbf{span} of a sequence of lines is the smallest subspace containing them, and a sequence of lines is \textbf{linearly dependent} if one lies in the span of the others. Equivalently, $n$ lines are linearly dependent if their span has dimension $d < n$.)
Given a sequence of lines $\{\ell_i\}_{i\in[n]}$, select a unit-norm representative $\varphi_i$ of each line $\ell_i$ and then compute the Gram matrix $\mathcal{G}=(\langle \varphi_j,\varphi_i\rangle)_{ij}$.
We say the lines $\{\ell_i\}_{i\in[n]}$ are \textbf{equiangular} if this Gram matrix enjoys a decomposition of the form
\begin{equation}
\label{eq.gram signature}
\mathcal{G}
=I+\mu\mathcal{S},
\end{equation}
where $\mu>0$ and $\mathcal{S}\in\mathbb{C}^{n\times n}$ is a \textbf{signature matrix}, that is, a self-adjoint matrix with zeros on the diagonal and unit-modulus entries off the diagonal.
Notice that every ETF produces equiangular lines.
Every signature matrix $\mathcal{S}$ is nonzero with zero trace, and so its minimum eigenvalue is negative.
In fact, for every signature matrix $\mathcal{S}$, one may take $\mu=-1/\lambda_\mathrm{min}(\mathcal{S})$ in~\eqref{eq.gram signature} to obtain the unique corresponding Gram matrix of unit-norm representatives of linearly dependent equiangular lines.
This correspondence between equiangular lines and signature matrices was observed by van Lint and Seidel~\cite{vanLintS:66}.

Note that a different choice of unit-norm representatives would lead to a different signature matrix.
In the real case, the off-diagonal entries in the signature matrix are discrete, lying in $\{\pm1\}$ instead of the entire complex unit circle $\mathbb{T}$, and this feature suggests a combinatorial description.
In order to elaborate, we need a few definitions:
We say vector sequences $\{\varphi_i\}_{i\in[n]}$ and $\{\psi_i\}_{i\in[n]}$ are \textbf{switching equivalent} if there exists $Q\in\operatorname{U}(d)$ and $\{\omega_i\}_{i\in[n]}$ in $\mathbb{T}$ such that $\psi_i=\omega_i Q\varphi_i$ for every $i\in[n]$.
Switching equivalence classes of unit-norm representatives of linearly dependent real equiangular lines are in one-to-one correspondence with combinatorial objects known as \textit{two-graphs} \cite{Seidel:76}.
A \textit{regular two-graph} corresponds to equiangular lines that arise from a real equiangular tight frame.

The fundamental problem in this area concerns the maximum number of equiangular lines with parameter $\mu>0$ that can reside in $\mathbb{C}^d$.
For this problem, the \textbf{relative bound}~\cite{vanLintS:66} states that
\[
n
\leq\frac{d(1-\mu^2)}{1-d\mu^2},
\]
provided $\mu<1/\sqrt{d}$.
Furthermore, equality is achieved precisely when the lines are spanned by vectors from an ETF, or equivalently, when the signature matrix of any choice of unit-norm representatives has exactly two eigenvalues.

\subsubsection{Association schemes}
\label{subsubsec.assoc}

Next, we review association schemes; see \cite{BannaiI:84,CecheriniST:08} for a complete treatment.
An \textbf{association scheme} is a sequence $\{A_i\}_{i\in[k]}$ in $\mathbb{C}^{n\times n}$ with entries in $\{0,1\}$ such that
\begin{itemize}
\item[(A1)]
$A_1=I$,
\item[(A2)]
$\sum_{i\in[k]}A_i=J$ (the matrix of all ones), and
\item[(A3)]
$\mathscr{A}:=\operatorname{span}\{A_i\}_{i\in[k]}$ is a $*$-algebra under matrix multiplication.
\end{itemize}
We refer to $\mathscr{A}$ as the \textbf{adjacency algebra} of $\{A_i\}_{i\in[k]}$; this is also known as a \textit{Bose--Mesner algebra}.
We say two association schemes are \textbf{isomorphic} if there exists a permutation matrix $P$ such that conjugating the adjacency matrices from one scheme by $P$ produces the adjacency matrices of the other scheme.
An association scheme is said to be \textbf{commutative} if its adjacency algebra is commutative.
In this case, the spectral theorem affords $\mathscr{A}$ with an alternative orthogonal basis of primitive idempotents, which can be combined to produce every orthogonal projection matrix in $\mathscr{A}$.
As such, if a commutative association scheme's $k$-dimensional adjacency algebra contains the Gram matrix of an ETF, then it can be obtained by searching through all $2^k$ combinations of the primitive idempotents.
This correspondence between association schemes and desirable Gram matrices dates back to Delsarte, Goethals and Seidel~\cite{DelsarteGS:77}, who coined the following phrase:
We say a matrix $M$ \textbf{carries the association scheme} $\{A_i\}_{i\in[k]}$ if $M=\sum_{i\in[k]}c_iA_i$ with $\{c_i\}_{i\in[k]}$ distinct (in words, the $A_i$'s indicate ``level sets'' of $M$).

A scheme is called \textbf{thin} if all of its adjacency matrices are permutation matrices, in which case the scheme is a permutation representation of a group $G$, and its adjacency algebra is isomorphic to the group ring $\mathbb{C}[G]$.
For example, the Cayley representation of the cyclic group $C_n$ produces a commutative association scheme of translation matrices whose adjacency algebra is the set of $n\times n$ circulant matrices.
For any association scheme, the adjacency matrices that are permutation matrices form a group known as the \textbf{thin radical}.

Since we are interested in doubly transitive lines, we expect unit-norm representatives of these lines to have a Gram matrix that exhibits additional algebraic structure.
Given a group $G$ acting transitively on a set $X$, we may consider the $*$-algebra of \textbf{$G$-stable} matrices, that is, matrices $M\in\mathbb{C}^{X\times X}$ satisfying $M_{g\cdot x,g\cdot y}=M_{x,y}$ for every $x,y\in X$ and $g\in G$.
Almost every member of this algebra carries an underlying association scheme, known as a \textit{Schurian scheme}; see the definition below.
To express the scheme's adjacency matrices, fix a point $x_0\in X$ and let $H$ denote the stabilizer of $x_0$ in $G$.
Since $G$ acts transitively on $X$, we may identify $X$ with $G/H$, as $g\cdot x_0$ corresponds to $gH$.
The group $G$ can be partitioned into \textbf{double cosets} in $H\backslash G/H$, defined by
\[
HaH
:=\{hah':h,h'\in H\},
\]
and each double coset can be further partitioned into left cosets.
These double cosets determine the adjacency matrices for the adjacency algebra $\mathscr{A}(G,H)$ of $G$-stable matrices with indices in $G/H$:
\begin{equation}
\label{eq.schurian adjacency matrices}
(A_{HaH})_{xH,yH}
=\left\{\begin{array}{ll}
1&\text{if }y^{-1}x\in HaH;\\
0&\text{otherwise}.
\end{array}\right.
\end{equation}
Any association scheme that arises in this way is called \textbf{Schurian}.
We say $(G,H)$ is a \textbf{Gelfand pair} if the $*$-algebra $\mathscr{A}(G,H)$ is commutative.

Throughout, it will be convenient to exploit other algebras that are isomorphic to $\mathscr{A}(G,H)$.
For example, consider the space
\[
L^2(H\backslash G /H)
:=\{f\colon G\to\mathbb{C}:f(gh)=f(hg)=f(g)\text{ for every }g\in G, h\in H\}
\]
of bi-$H$-invariant functions on $G$.
Equivalently, these are complex-valued functions over $G$ that are constant on double cosets of $H$, namely, members of the span of the indicator functions $\mathbf{1}_{HaH}$ for $a\in G$, where
\[
\mathbf{1}_{S}(g)
:=\left\{\begin{array}{cl}
1&\text{if }g\in S;\\
0&\text{otherwise,}
\end{array}\right.
\qquad
(S\subseteq G).
\]
This vector space is a $*$-algebra with convolution and involution:
\[
(f_1*f_2)(g)
=\sum_{h\in G}f_1(h)f_2(h^{-1}g),
\qquad
f^*(g)=\overline{f(g^{-1})},
\qquad
(g\in G).
\]
Furthermore, the mapping $\phi\colon\mathscr{A}(G,H)\to L^2(H\backslash G/H)$ defined by $(\phi(M))(g)=\frac{1}{|H|}M_{gH,H}$ is a $*$-algebra anti-isomorphism; here, ``anti'' indicates that the mapping switches the order of multiplication: $\phi(AB)=\phi(B)*\phi(A)$.
Next, $L^2(H\backslash G/H)$ embeds into the group ring by $\theta\colon L^2(H\backslash G/H)\to\mathbb{C}[G]$ defined by $\theta(f)=\sum_{g\in G}f(g)g$.
(Here and throughout, we identify $G$ and $\mathbb{C}$ with their images in the group ring $\mathbb{C}[G]$; in cases where this invites confusion, notably, when $G\subseteq\mathbb{C}$, we instead use the notation $\delta_g\in\mathbb{C}[G]$ for the image of $g\in G$, though we note that the notation $\underline{g}$ is also common.)
The range of $\theta$ is
\[
\mathbb{C}[H\backslash G /H]
:=\bigg\{\sum_{S\in H\backslash G /H}\sum_{g\in S}c_{S}g:c_{S}\in\mathbb{C}\text{ for every }S\in H\backslash G /H\bigg\}.
\]
In particular, $\mathbb{C}[H\backslash G /H]$ is also a $*$-algebra with the usual group ring multiplication and with involution defined by
\[
\bigg(\sum_{S\in H\backslash G /H}\sum_{g\in S}c_{S}g\bigg)^*
=\sum_{S\in H\backslash G /H}\sum_{g\in S}\overline{c_{S}}g^{-1}.
\]
As such, $\theta$ is a $*$-algebra isomorphism.
To summarize, we have two $*$-algebra (anti-) isomorphisms available for our use:
\begin{equation}
\label{eq.isomorphisms}
\mathscr{A}(G,H)
\quad
\stackrel{\phi}{\longrightarrow}
\quad
L^2(H\backslash G/H)
\quad
\stackrel{\theta}{\longrightarrow}
\quad
\mathbb{C}[H\backslash G/H].
\end{equation}

\subsection{Motivating example and main results}
\label{subsec.example main results}

We start with an example of four lines in $\mathbb{C}^2$.
Recall that lines in $\mathbb{C}^2$ correspond to points in complex projective space $\mathbb{C}\mathbf{P}^1$, which as we show below is isometric to the unit sphere $S^2$ (this correspondence is known as the Bloch sphere~\cite{NielsenC:00} in quantum mechanics).
As such, we expect symmetric collections of lines through the origin to correspond to symmetric collections of points in the sphere.
Since we want four lines in $\mathbb{C}^2$, we are naturally drawn to the vertices of a regular tetrahedron circumscribed by $S^2$.
In fact, these lines are doubly transitive:
The action of $\operatorname{U}(2)$ on $\mathbb{C}\mathbf{P}^1$ corresponds to $\operatorname{SO}(3)$ acting on $S^2$, and it is easy to convince oneself that $\operatorname{SO}(3)$ acts doubly transitively on these vertices (especially with the help of a four-sided die).
Explicitly, the isometry we are leveraging is induced by $f\colon \varphi\mapsto \sqrt{2}(\varphi\varphi^*-\frac{1}{2}I)$, which maps unit vectors in $\mathbb{C}^2$ into the $3$-dimensional real vector space of $2\times 2$ self-adjoint matrices with zero trace. (Indeed, each fibre of $f$ consists of all unit-norm representatives of a common line, and modding out by this equivalence produces an isometry $\mathbb{C} \mathbf{P}^1 \cong S^2$.)
This mapping interacts nicely with inner products:
\[
\langle f(\varphi),f(\psi)\rangle_\mathrm{HS}
=2|\langle \varphi,\psi\rangle|^2-1.
\]
(Here, $\langle A,B\rangle_\mathrm{HS}:=\operatorname{tr}(AB^*)$ denotes the Hilbert--Schmidt inner product.)
Consider unit-norm representatives of our doubly transitive lines, that is, $\{\varphi_i\}_{i\in[4]}$ in $\mathbb{C}^2$ so that $\{f(\varphi_i)\}_{i\in[4]}$ form the vertices of a regular tetrahedron.
We can use the mapping $f$ to show that $\{\varphi_i\}_{i\in[4]}$ forms an equiangular tight frame for $\mathbb{C}^2$.
First, the vertices sum to zero, and so
\[
\sum_{i\in[4]}\varphi_i\varphi_i^*
=\sum_{i\in[4]}\bigg(\frac{1}{\sqrt{2}}f(\varphi_i)+\frac{1}{2}I\bigg)
=2I.
\]
Next, when $i\neq j$, we have $\langle f(\varphi_i),f(\varphi_j)\rangle_\mathrm{HS}=-1/3$ and so
\[
|\langle\varphi_i,\varphi_j\rangle|^2
=\frac{1}{2}\Big(\langle f(\varphi_i),f(\varphi_j)\rangle_\mathrm{HS}+1\Big)
=\frac{1}{3}.
\]
The fact that an ETF arose from highly symmetric lines is no coincidence (see also~\cite{Creignou:07}):

\begin{lemma}
\label{lem.2tranETF}
Given $n$ doubly transitive lines with span $\mathbb{C}^d$, select unit-norm representatives $\{\varphi_i\}_{i\in[n]}$.
\begin{itemize}
\item[(a)]
There exists $\beta$ such that $|\langle \varphi_i,\varphi_j\rangle|^2=\beta$ for every $i,j\in[n]$ with $i\neq j$.
\item[(b)]
If $n>d$, then there exists $\alpha$ such that $\sum_{i\in[n]}\varphi_i\varphi_i^*=\alpha I$.
\end{itemize}
\end{lemma}

We will prove this lemma shortly.
First, we note that part (a) does not require $n$ to be finite, and in fact, part (a) implies that $n$ is finite; indeed, Gerzon's bound~\cite{LemmensS:73} gives that $n$ lines are equiangular only if $n\leq d^2$.
For part (b), the requirement $n>d$ is important:
If $n=d$, then given an orthonormal basis $\{e_i\}_{i\in[d]}$ for $\mathbb{C}^d$, define $s=\sum_{j\in[d]}e_j$ and $\varphi_i=e_i+s$ for every $i\in[n]$; the lines spanned by $\{\varphi_i\}_{i\in[n]}$ are doubly transitive (in fact, the automorphism group is all of $S_n$), but
\[
\sum_{i\in[n]}\varphi_i\varphi_i^*
=\sum_{i\in[n]}(e_i+s)(e_i+s)^*
=\sum_{i\in[n]}e_ie_i^*+3ss^*
=I+3ss^*,
\]
which has two distinct eigenvalues, unlike $\alpha I$.
Overall, we have that doubly transitive lines with $n>d$ necessarily produce ETFs.
(Recall that the previous example had $n=4>2=d$.)

For the proof of Lemma~\ref{lem.2tranETF}, it is convenient to pass the notion of double transitivity to the unit-norm representatives.
To this end, we consider the \textbf{projective symmetry group} of $\{\varphi_i\}_{i\in[n]}$, defined to be the group of permutations $\sigma\in S_n$ for which there exist $Q\in \operatorname{U}(d)$ and phases $\{\omega_i\}_{i\in[n]}$ such that $Q\varphi_i=\omega_i\varphi_{\sigma(i)}$ for every $i\in[n]$.
The automorphism group of a sequence of lines is identical to the projective symmetry group of any choice of unit-norm representatives.

\begin{proof}[Proof of Lemma~\ref{lem.2tranETF}]
For (a), take $a,b,a',b'\in[n]$ with $a\neq b$ and $a'\neq b'$.
Then by double transitivity, there exists $\sigma$ in the projective symmetry group of $\{\varphi_i\}_{i\in[n]}$ that maps $a\mapsto a'$ and $b\mapsto b'$.
Letting $Q$ and $\{\omega_i\}_{i\in[n]}$ denote the corresponding unitary and phases, this in turn implies
\[
|\langle \varphi_a,\varphi_b\rangle|^2
=|\langle Q\varphi_a,Q\varphi_b\rangle|^2
=|\langle \omega_a\varphi_{a'},\omega_b\varphi_{b'}\rangle|^2
=|\langle \varphi_{a'},\varphi_{b'}\rangle|^2.
\]
Since our choice for $a,b,a',b'\in[n]$ was arbitrary, we may conclude equiangularity.

For (b), let $\mathcal{G}$ denote the Gram matrix of $\{\varphi_i\}_{i\in[n]}$, whose $(i,j)$th entry is given by $\langle \varphi_j,\varphi_i\rangle$.
Then, borrowing notation from the proof of (a), we have
\[
\mathcal{G}_{ab}
=\langle\varphi_b,\varphi_a\rangle
=\langle Q\varphi_b,Q\varphi_a\rangle
=\langle\omega_b\varphi_{b'},\omega_a\varphi_{a'}\rangle
=\overline{\omega_a}\omega_b \langle\varphi_{b'},\varphi_{a'}\rangle
=\overline{\omega_a}\omega_b \mathcal{G}_{a'b'},
\]
and furthermore,
\begin{align*}
(\mathcal{G}^2)_{ab}
&=\sum_{i\in[n]} \langle \varphi_i,\varphi_a\rangle\langle \varphi_b,\varphi_i\rangle
=\sum_{i\in[n]} \langle Q\varphi_i,Q\varphi_a\rangle\langle Q\varphi_b,Q\varphi_i\rangle\\
&=\sum_{i\in[n]} \langle \omega_i\varphi_{\sigma(i)},\omega_a\varphi_{a'}\rangle\langle \omega_b\varphi_{b'},\omega_i\varphi_{\sigma(i)}\rangle
=\overline{\omega_a}\omega_b\sum_{i\in[n]}\langle \varphi_{\sigma(i)},\varphi_{a'}\rangle\langle \varphi_{b'},\varphi_{\sigma(i)}\rangle
=\overline{\omega_a}\omega_b(\mathcal{G}^2)_{a'b'}.
\end{align*}
As such, the off-diagonal of $\mathcal{G}^2$ is a constant multiple of the off-diagonal of $\mathcal{G}$.
Moreover,
\[
(\mathcal{G}^2)_{aa}
=\sum_{i\in[n]}|\langle \varphi_i,\varphi_a\rangle|^2
=1+(n-1)\beta
=\big(1+(n-1)\beta\big)\mathcal{G}_{aa}.
\]
Overall, $\mathcal{G}^2=c_1\mathcal{G}+c_2I$ for some $c_1,c_2\in\mathbb{C}$, and so every eigenvalue $\lambda$ of $\mathcal{G}$ satisfies $\lambda^2=c_1\lambda+c_2$.
Since $n>d$ by assumption, $\mathcal{G}$ is rank-deficient, meaning $\lambda=0$ is an eigenvalue of $\mathcal{G}$, and so $c_2=0$.
As such, $\mathcal{G}$ is a scalar multiple of an orthogonal projection matrix, which gives the result.
\end{proof}

The above proof exploits how ETFs are easily characterized in terms of the Gram matrix, i.e., it is equivalent for the Gram matrix to be a scalar multiple of an orthogonal projection matrix whose off-diagonal entries all have the same modulus.
This characterization interacts nicely with the theory of association schemes, especially those arising from Gelfand pairs~\cite{IversonJM:16,IversonJM:17}.
In the study of doubly transitive lines, a particular type of Gelfand pair is especially relevant.
We name the following object after a pair of mathematicians, namely, Graham Higman and Donald G.\ Higman, who are known for their contributions to the theory of groups, two-graphs, and association schemes~\cite{Collins:08,BannaiGPS:09}.

\begin{definition}
Given a finite group $G$ and a proper subgroup $H\leq G$, let $K=N_G(H)$ be the normalizer of $H$ in $G$.
We say $(G,H)$ is a \textbf{Higman pair} if there exists a \textbf{key} $b\in G\setminus K$ such that
\begin{itemize}
\item[(H1)]
$G$ acts doubly transitively on $G/K$,
\item[(H2)]
$K/H$ is abelian,
\item[(H3)]
$HbH=Hb^{-1}H$,
\item[(H4)]
$aba^{-1}\in HbH$ for every $a\in K$, and
\item[(H5)]
$a \in K$ satisfies $ab \in HbH$ only if $a \in H$.
\end{itemize}
\end{definition}

As an example, consider the isomorphism $\beta\colon\mathbb{F}_3^*\to C_2$ and take
\[
G:=\operatorname{SL}(2,3)\times C_4,
\qquad
H:=\{([\begin{smallmatrix}x&y\\0&x\end{smallmatrix}],\beta(x)):x,y\in\mathbb{F}_3,x\neq0\}.
\]
It turns out that $(G,H)$ is a Higman pair with
\[
K=\{([\begin{smallmatrix}x&y\\0&x\end{smallmatrix}],z):x,y\in\mathbb{F}_3,x\neq0,z\in C_4\},
\qquad
b=([\begin{smallmatrix}0&1\\-1&0\end{smallmatrix}],\mathrm{i}).
\]
(Here and throughout, we view $C_r$ as the subgroup of $\mathbb{C}$ comprised of $r$th roots of unity, and we denote $\mathrm{i}=\sqrt{-1}$.)
We will only verify (H1) here, as the proofs of $K=N_G(H)$ and (H2)--(H5) are short and unenlightening.
Since $\operatorname{SL}(2,3)$ permutes the set $X$ of one-dimensional subspaces of $\mathbb{F}_3^2$, we may let $G$ act on $X$ by setting $(g,z)\cdot x = g\cdot x$.
Then since $\operatorname{SL}(2,3)$ acts doubly transitively on $X$, $G$ does, as well.
Now observe that $K$ is the stabilizer of the line through $[1,0]^\top$, meaning the action of $G$ on $G/K$ is equivalent to that on $X$.
This gives (H1).

With the help of GAP~\cite{GAP:software,Hanaki:online}, one can show that the algebra $\mathscr{A}(G,H)$ has a basis of eight $16\times 16$ adjacency matrices: four of the form $D^j$ and four of the form $D^jA$, where
\[
D
=\left[\begin{array}{cccc}
T&\cdot&\cdot&\cdot\\
\cdot&T&\cdot&\cdot\\
\cdot&\cdot&T&\cdot\\
\cdot&\cdot&\cdot&T
\end{array}\right],
\qquad
A
=\left[\begin{array}{llll}
\phantom{|}\cdot&T^{\phantom{-1}} &T^{\phantom{-1}} &T^{\phantom{-1}}\\
T^{-1}&\phantom{|}\cdot&T^{\phantom{-1}}&T^{-1}\\
T^{-1}&T^{-1}&\phantom{|}\cdot&T^{\phantom{-1}}\\
T^{-1}&T^{\phantom{-1}}&T^{-1}&\phantom{|}\cdot
\end{array}\right],
\qquad
T
=\left[\begin{array}{cccc}
\cdot&\cdot&\cdot&1\\
1&\cdot&\cdot&\cdot\\
\cdot&1&\cdot&\cdot\\
\cdot&\cdot&1&\cdot
\end{array}\right].
\]
Here, dots denote zeros and $T$ is the Cayley representation of $\mathrm{i}\in C_4$.
It is straightforward to verify that $(G,H)$ is a Gelfand pair, and so $\mathscr{A}(G,H)$ contains eight primitive idempotents.
While this determines $2^8=256$ different orthogonal projection matrices, it turns out that in this case, the primitive idempotents already yield interesting Gram matrices.
Of these, two have rank $1$, four have rank $2$, and the remaining two have rank $3$.
One of the rank-$2$ idempotents is given below:
\[
\begingroup 
\setlength\arraycolsep{2pt}
P
=\frac{1}{8\sqrt{3}}
\tiny{
\left[\begin{array}{rrrr|rrrr|rrrr|rrrr}
\cellcolor{lightgray} \sqrt{3}&-\sqrt{3}\mathrm{i}&-\sqrt{3}&\sqrt{3}\mathrm{i}&\cellcolor{lightgray} \mathrm{i}&1&-\mathrm{i}&-1&\cellcolor{lightgray} \mathrm{i}&1&-\mathrm{i}&-1&\cellcolor{lightgray} \mathrm{i}&1&-\mathrm{i}&-1\\
\sqrt{3}\mathrm{i}&\sqrt{3}&-\sqrt{3}\mathrm{i}&-\sqrt{3}&-1&\mathrm{i}&1&-\mathrm{i}&-1&\mathrm{i}&1&-\mathrm{i}&-1&\mathrm{i}&1&-\mathrm{i}\\
-\sqrt{3}&\sqrt{3}\mathrm{i}&\sqrt{3}&-\sqrt{3}\mathrm{i}&-\mathrm{i}&-1&\mathrm{i}&1&-\mathrm{i}&-1&\mathrm{i}&1&-\mathrm{i}&-1&\mathrm{i}&1\\
-\sqrt{3}\mathrm{i}&-\sqrt{3}&\sqrt{3}\mathrm{i}&\sqrt{3}&1&-\mathrm{i}&-1&\mathrm{i}&1&-\mathrm{i}&-1&\mathrm{i}&1&-\mathrm{i}&-1&\mathrm{i}\\\hline
\cellcolor{lightgray}-\mathrm{i}&-1&\mathrm{i}&1&\cellcolor{lightgray} \sqrt{3}&-\sqrt{3}\mathrm{i}&-\sqrt{3}&\sqrt{3}\mathrm{i}&\cellcolor{lightgray} \mathrm{i}&1&-\mathrm{i}&-1&\cellcolor{lightgray} -\mathrm{i}&-1&\mathrm{i}&1\\
1&-\mathrm{i}&-1&\mathrm{i}&\sqrt{3}\mathrm{i}&\sqrt{3}&-\sqrt{3}\mathrm{i}&-\sqrt{3}&-1&\mathrm{i}&1&-\mathrm{i}&1&-\mathrm{i}&-1&\mathrm{i}\\
\mathrm{i}&1&-\mathrm{i}&-1&-\sqrt{3}&\sqrt{3}\mathrm{i}&\sqrt{3}&-\sqrt{3}\mathrm{i}&-\mathrm{i}&-1&\mathrm{i}&1&\mathrm{i}&1&-\mathrm{i}&-1\\
-1&\mathrm{i}&1&-\mathrm{i}&-\sqrt{3}\mathrm{i}&-\sqrt{3}&\sqrt{3}\mathrm{i}&\sqrt{3}&1&-\mathrm{i}&-1&\mathrm{i}&-1&\mathrm{i}&1&-\mathrm{i}\\\hline
\cellcolor{lightgray} -\mathrm{i}&-1&\mathrm{i}&1&\cellcolor{lightgray} -\mathrm{i}&-1&\mathrm{i}&1&\cellcolor{lightgray} \sqrt{3}&-\sqrt{3}\mathrm{i}&-\sqrt{3}&\sqrt{3}\mathrm{i}&\cellcolor{lightgray} \mathrm{i}&1&-\mathrm{i}&-1\\
1&-\mathrm{i}&-1&\mathrm{i}&1&-\mathrm{i}&-1&\mathrm{i}&\sqrt{3}\mathrm{i}&\sqrt{3}&-\sqrt{3}\mathrm{i}&-\sqrt{3}&-1&\mathrm{i}&1&-\mathrm{i}\\
\mathrm{i}&1&-\mathrm{i}&-1&\mathrm{i}&1&-\mathrm{i}&-1&-\sqrt{3}&\sqrt{3}\mathrm{i}&\sqrt{3}&-\sqrt{3}\mathrm{i}&-\mathrm{i}&-1&\mathrm{i}&1\\
-1&\mathrm{i}&1&-\mathrm{i}&-1&\mathrm{i}&1&-\mathrm{i}&-\sqrt{3}\mathrm{i}&-\sqrt{3}&\sqrt{3}\mathrm{i}&\sqrt{3}&1&-\mathrm{i}&-1&\mathrm{i}\\\hline
\cellcolor{lightgray} -\mathrm{i}&-1&\mathrm{i}&1&\cellcolor{lightgray} \mathrm{i}&1&-\mathrm{i}&-1&\cellcolor{lightgray} -\mathrm{i}&-1&\mathrm{i}&1&\cellcolor{lightgray} \sqrt{3}&-\sqrt{3}\mathrm{i}&-\sqrt{3}&\sqrt{3}\mathrm{i}\\
1&-\mathrm{i}&-1&\mathrm{i}&-1&\mathrm{i}&1&-\mathrm{i}&1&-\mathrm{i}&-1&\mathrm{i}&\sqrt{3}\mathrm{i}&\sqrt{3}&-\sqrt{3}\mathrm{i}&-\sqrt{3}\\
\mathrm{i}&1&-\mathrm{i}&-1&-\mathrm{i}&-1&\mathrm{i}&1&\mathrm{i}&1&-\mathrm{i}&-1&-\sqrt{3}&\sqrt{3}\mathrm{i}&\sqrt{3}&-\sqrt{3}\mathrm{i}\\
-1&\mathrm{i}&1&-\mathrm{i}&1&-\mathrm{i}&-1&\mathrm{i}&-1&\mathrm{i}&1&-\mathrm{i}&-\sqrt{3}\mathrm{i}&-\sqrt{3}&\sqrt{3}\mathrm{i}&\sqrt{3}\\
\end{array}\right]
}.
\endgroup
\]
(We will make use of the shaded entries later.)
Multiplying $P$ by $8$ gives the Gram matrix of four unit-norm representatives from each of four lines, and furthermore, any vectors $\varphi_i$ and $\varphi_j$ representing different lines satisfy $|\langle \varphi_i,\varphi_j\rangle|^2=1/3$.
Indeed, selecting a single representative from each line produces an ETF of four vectors in $\mathbb{C}^2$, i.e., the ETF discussed at the beginning of this subsection.
In particular, these lines are doubly transitive.
This correspondence between Higman pairs and doubly transitive lines occurs in general:

\begin{theorem}[Higman Pair Theorem]\
\label{thm.Higman Pair Theorem}
\begin{itemize}
\item[${}^*$(a)]
Assume $n\geq 3$.
Given $n>d$ doubly transitive lines that span $\mathbb{C}^d$, there exists $r$ such that one may select $r$ equal-norm representatives from each of the $n$ lines whose Gram matrix carries the association scheme of a Higman pair $(G,H)$ with $r=[N_G(H):H]$ and $n=[G:N_G(H)]$.
Moreover, their Gram matrix is a primitive idempotent for this scheme.
\item[(b)]
Every Higman pair $(G,H)$ is a Gelfand pair.
Every primitive idempotent of its association scheme is the Gram matrix of $r=[N_G(H):H]$ equal-norm representatives from each of $n=[G:N_G(H)]$ doubly transitive lines that span $\mathbb{C}^d$ with $d<n$, and the phase of each entry is an $r$th root of unity.
Moreover, the automorphism group of the lines contains the doubly transitive action of $G$ on $G/N_G(H)$.
\end{itemize}
\end{theorem}

We save the proof of Theorem~\ref{thm.Higman Pair Theorem}(a) for our sequel paper~\cite{IversonM:future}, as the techniques in this proof are very different from the themes in this paper, and furthermore, the proof presents our method for classifying doubly transitive lines, which we also perform in~\cite{IversonM:future}.
We include an asterisk in the theorem statement to mark this distinction.

As in the case of four lines in $\mathbb{C}^2$, doubly transitive lines always exhibit the remarkable feature that, for some finite $r$, one may select $r$ unit-norm representatives from each line in such a way that the phase of every inner product is an $r$th root of unity.
Next, we note that the block form of $D$ and $A$ above suggest that we embed the adjacency algebra $\mathscr{A}(G,H)$ as a subalgebra of $\mathbb{C}[C_4]^{4\times 4}$.
Under this mapping, $\{D^j\}_{j\in[4]}$ is sent to $\{\delta_gI\}_{g\in C_4}$ and $\{D^jA\}_{j\in[4]}$ is sent to $\{\delta_gB\}_{g\in C_4}$, where
\begin{equation}
\label{eq.roux example}
B
=\left[\begin{array}{llll}
0&\delta_{\mathrm{i}\phantom{-}} &\delta_{\mathrm{i}\phantom{-}} &\delta_{\mathrm{i}\phantom{-}}\\
\delta_{-\mathrm{i}}&0&\delta_{\mathrm{i}\phantom{-}}&\delta_{-\mathrm{i}}\\
\delta_{-\mathrm{i}}&\delta_{-\mathrm{i}}&0&\delta_{\mathrm{i}\phantom{-}}\\
\delta_{-\mathrm{i}}&\delta_{\mathrm{i}\phantom{-}}&\delta_{-\mathrm{i}}&0
\end{array}\right].
\end{equation}
(Recall that $\delta_g\in\mathbb{C}[C_4]$ denotes the image of $g\in C_4$ in the group ring.)
Note that the embedding $\mathscr{A}(G,H)\to\mathbb{C}[C_4]^{4\times 4}$ can be inverted by applying the Cayley representation (or more precisely, its linear extension to the group ring) to each matrix entry.
It is convenient to formalize the role that $B$ plays here:

\begin{definition}
A \textbf{roux} for a multiplicative abelian group $\Gamma$ is an $n\times n$ matrix $B$ with entries in $\mathbb{C}[\Gamma]$ such that each of the following holds simultaneously:
\begin{itemize}
\item[(R1)]
$B_{ii}=0$ for every $i\in[n]$.
\item[(R2)]
$B_{ij}\in\Gamma$ for every $i,j\in[n]$, $i\neq j$.
\item[(R3)]
$B_{ji}=(B_{ij})^{-1}$ for every $i,j\in[n]$, $i\neq j$.
\item[(R4)]
The matrices $\{gI\}_{g\in\Gamma}$ and $\{gB\}_{g\in\Gamma}$ span an algebra $\mathscr{A}(B)$.
\end{itemize}
\end{definition}

Note that $\mathscr{A}(B)$ is necessarily a commutative $*$-algebra since
\[
(gI)^*=g^{-1}I\in\mathscr{A}(B),
\qquad
(gB)^*=g^{-1}B\in\mathscr{A}(B),
\]
and the $gI$'s and $gB$'s all commute.
Later, we will explain how roux generalize the theory of regular abelian distance-regular antipodal covers of the complete graph (regular abelian \textsc{drackn}s), as studied by Godsil and Hensel in~\cite{GodsilH:92}.

Given a roux for $\Gamma$, we may evaluate the roux at a linear character $\alpha$ of $\Gamma$.
Specifically, $\alpha$ extends linearly to $\mathbb{C}[\Gamma]$, and its entrywise application amounts to a $*$-algebra homomorphism $\mathbb{C}[\Gamma]^{n\times n}\to\mathbb{C}^{n\times n}$.
Evaluating a roux at a character produces a signature matrix, and hence, equiangular lines.
For example, evaluating the above roux for $C_4$ at the character $\alpha$ defined by $\alpha(z)=z$ gives
\begin{equation}
\label{eq.2by4 signature matrix}
\mathcal{S}
=\left[\begin{array}{rrrr}
0&\mathrm{i} &\mathrm{i} &\mathrm{i}\\
-\mathrm{i}&0&\mathrm{i}&-\mathrm{i}\\
-\mathrm{i}&-\mathrm{i}&0&\mathrm{i}\\
-\mathrm{i}&\mathrm{i}&-\mathrm{i}&0
\end{array}\right].
\end{equation}
Since $\lambda_\mathrm{min}(\mathcal{S})=-\sqrt{3}$, the corresponding Gram matrix is $\mathcal{G}=I+(1/\sqrt{3})S$, which happens to be a principal submatrix of $P$ (if we ignore the additional factor of $1/8$ in $P$), namely, the shaded submatrix indexed by $\{1,5,9,13\}$.
We define \textbf{roux lines} to be any sequence of linearly dependent lines for which there exist unit-norm representatives whose signature matrix can be obtained by evaluating a roux at a character.

\begin{corollary}\
\label{cor.2tran-roux-etf}
\begin{itemize}
\item[${}^*$(a)]
All doubly transitive lines are roux.
\item[(b)]
Every roux line sequence has unit-norm representatives that form an equiangular tight frame for their span.
\end{itemize}
\end{corollary}


Our proof of Corollary~\ref{cor.2tran-roux-etf}(a) uses Theorem~\ref{thm.Higman Pair Theorem}(a) along with a characterization of association schemes that arise from Higman pairs; see Theorem~\ref{thm.schurian roux}.
We mark Corollary~\ref{cor.2tran-roux-etf}(a) with an asterisk above due to its dependence on Theorem~\ref{thm.Higman Pair Theorem}(a), whose proof does not appear in this paper.

Theorem~\ref{thm.Higman Pair Theorem} and Corollary~\ref{cor.2tran-roux-etf} together imply that roux generalize doubly transitive lines.
Explicitly, any sequence of doubly transitive lines has unit-norm representatives whose Gram matrix carries a Schurian association scheme that satisfies a few axioms.
Conversely, any Schurian association scheme satisfying these axioms produces doubly transitive lines through its primitive idempotents.
By generalizing these axioms to non-Schurian schemes, we produce the notion of a \textit{roux scheme}, which in turn determines a roux.
Even in the non-Schurian case, the primitive idempotents of a roux scheme describe equiangular lines (indeed, ETFs) in complex space, and it turns out that these are precisely the corresponding roux lines.

This paper is devoted to a detailed study of roux and related geometric and combinatorial objects.
In addition to roux lines and roux schemes, a particular adjacency matrix of a roux scheme determines a \textit{roux graph}, which in turn generalizes the existing notion of a regular abelian \textsc{drackn}~\cite{GodsilH:92,CoutinhoCSZ:16}.
These relationships and others make up the main results of this paper, which are summarized in Figure~\ref{figure.diagram}.

\begin{figure}
\begin{center}
\begin{tikzpicture}
\node at (0,7.3){\fbox{\textbf{\phantom{y}Geometry\phantom{y}}}};
\node at (11,7.3){\fbox{\textbf{\phantom{y}Combinatorics\phantom{y}}}};
\node at (0,6){doubly transitive lines};
\node at (5.5,6){Higman pair};
\node at (11,6){Schurian roux scheme};
\node at (0,4){roux lines};
\node at (5.5,4){roux};
\node at (11,4){roux scheme};
\node at (0,2){ETF};
\node at (11,2){roux graph};
\node at (0,0){\textsc{drackn} lines};
\node at (11,0){regular abelian \textsc{drackn}};
\draw [>=stealth,->] (2.2,6.1)--({3.6+0.5},6.1);
\node at ({(2.2+3.6+0.5)/2},6.3){\tiny{Thm.~\ref{thm.Higman Pair Theorem}(a)$^*$}};
\draw [>=stealth,<-] (2.2,5.9)--({3.6+0.5},5.9);
\node at ({(2.2+3.6+0.5)/2},5.7){\tiny{Thm.~\ref{thm.Higman Pair Theorem}(b)}};
\draw [>=stealth,<->] ({6.4+0.5},6)--({7.8+1},6);
\node at ({(6.4+0.5+7.8+1)/2},6.2){\tiny{Thm.~\ref{thm.schurian roux}}};
\draw [>=stealth,->] (0,5.6)--(0,4.4);
\node at (-0.8,{(5.6+4.4)/2}){\tiny{Cor.~\ref{cor.2tran-roux-etf}(a)$^*$}};
\draw [>=stealth,->] (11,5.6)--(11,4.4);
\draw [>=stealth,->] (1.2,4.1)--(4.7,4.1);
\node at ({(1.2+4.7)/2},4.3){\tiny{Cor.~\ref{cor.roux lines detector}}};
\draw [>=stealth,->] (4.7,3.9)--(1.2,3.9);
\draw [>=stealth,<->] (6.3,4)--(9.6,4);
\draw [>=stealth,->] (0,3.6)--(0,2.4);
\node at (-0.8,{(3.6+2.4)/2}){\tiny{Cor.~\ref{cor.2tran-roux-etf}(b)}};
\draw [>=stealth,<->] (11,3.6)--(11,2.4);
\draw [>=stealth,->] (0,0.4)--(0,1.6);
\node at (-1.1,{(0.4+1.6)/2}){\tiny{Thm.~4.1 in~\cite{CoutinhoCSZ:16}}};
\draw [>=stealth,->] (11,0.4)--(11,1.6);
\node at ({11-0.9},{(0.4+1.6)/2}){\tiny{Thm.~\ref{thm.drackn vs roux}(a)}};
\draw [>=stealth,->] (1.5,0.1)--(8.5,0.1);
\node at ({(1.5+8.5)/2},0.3){\tiny{Cor.~\ref{cor.drackn lines detector}}};
\draw [>=stealth,<-] (1.5,-0.1)--(8.5,-0.1);
\end{tikzpicture}
\end{center}
\caption{
\label{figure.diagram}
{\small 
Relationships between the primary objects in this paper.
Here, $A\to B$ indicates that an object of type $A$ gives rise to an object of type $B$, and unlabeled arrows follow from definition.
Notice that Higman pairs and roux serve as intermediaries between the worlds of geometry and combinatorics.
}\normalsize}
\end{figure}

\subsection{Context}
\label{subsec.context}

Our study of doubly transitive lines follows a tradition of pursuing optimality by way of symmetry~\cite{FejesToth:64}.
In fact, the specific connection we draw between symmetry and roux enjoys a historical precedent.
Taylor~\cite{Taylor:77} and Seidel~\cite{Seidel:92} report that G.~Higman introduced regular two-graphs in 1970 while studying doubly transitive permutation groups, specifically the action of $\text{Co}_3$ on 276 points.
As a group theorist, Higman may have been motivated by Burnside's theorem~\cite{Burnside:97}, which indicates a correspondence between certain doubly transitive actions and finite simple groups.
Evidently, Higman suspected that regular two-graphs could provide a setting for such groups, acting as their automorphisms.
In a sequel paper, we prove that doubly transitive two-graphs correspond to doubly transitive real lines~\cite{IversonM:future}.
In this sense, Higman was actually studying the automorphism groups of doubly transitive real lines when he naturally uncovered real ETFs, in their guise as regular two-graphs.
In direct analogy, the authors discovered roux while studying doubly transitive lines in complex space, and in fact, roux may be seen as combinatorial generalizations of doubly transitive lines.
Doubly transitive two-graphs have been classified by Taylor~\cite{Taylor:92}, and the present series may be seen as extending Taylor's work for the complex setting.
The connection between doubly transitive lines and equiangular tight frames was independently observed by Creignou~\cite{Creignou:07}.

In addition to regular two-graphs (which correspond with real ETFs), we will prove that roux encapsulate the notion of regular abelian \textsc{drackn}s.
The latter are already known to create ETFs~\cite{GodsilH:92,CoutinhoCSZ:16}, and the corresponding class of \textsc{drackn} lines contain all real ETFs, and are further contained by roux lines.
See Figure~\ref{figure.venn} for an illustration of these relationships.
The containment of regular abelian \textsc{drackn}s by roux is made possible by the generality of considering roux over arbitrary finite abelian groups, and not merely cyclic groups.
See Example~\ref{ex.thas somma drackn} for an explicit example of a regular abelian \textsc{drackn} (hence, a roux) that is most naturally defined over a non-cyclic group.

\begin{figure}
\begin{center}
\begin{tabular}{ll}
\begin{tikzpicture}[scale=0.8]
\coordinate (0) at (0,0);
\coordinate (1) at (1,0);
\coordinate (2) at (2,0);
\coordinate (3) at (1,3.95);
\coordinate (4) at (1,2.7);
\coordinate (5) at (0,1.37);
\coordinate (6) at (0,0);
\coordinate (7) at (2.9,0.3);
\coordinate (8) at (2.9,0);
\coordinate (9) at (2.9,-0.3);
\draw [thick] (1) circle [radius=4.5];
\draw [thick] (1) circle [radius=3.5];
\draw [thick] (0) circle [radius=2];
\draw [thick] (0) circle [radius=1];
\draw [dashed] (2) circle [radius=2];
\draw [fill=white, draw=white] (0.5,1.33) circle [radius=0.21];
\draw [fill=white, draw=white] (0,0) circle [radius=0.21];
\node at (3){\tiny{complex ETFs}};
\node at (4){\tiny{roux lines}};
\node at (5){\tiny{\textsc{drackn} lines}};
\node at (6){\tiny{real ETFs}};
\node at (7){\tiny{doubly}};
\node at (8){\tiny{transitive}};
\node at (9){\tiny{lines}};
\end{tikzpicture}
\end{tabular}
\end{center}
\caption{
\label{figure.venn}
{\small 
Venn diagram of the primary geometric objects in this paper.
While equiangular tight frames (ETFs) are not lines, they correspond to lines in a natural way.
All of these containments are nontrivial, and many of them correspond to results in this paper.
}\normalsize}
\end{figure}

These applications of roux and \textsc{drackn}s follow the well-established approach of using combinatorial designs to facilitate geometric constructions.
For instance, the technique of extracting unit-norm vectors from the idempotents of an association scheme dates back to the study of spherical designs by Delsarte, Goethels, and Seidel~\cite{DelsarteGS:77}, and has also been applied for ETFs~\cite{IversonJM:spie17,IversonJM:16,IversonJM:17}.
In addition to association schemes, ETFs have been constructed with the help of numerous other combinatorial methods, including difference sets~\cite{StrohmerH:03,XiaZG:05,DingF:07}, strongly regular graphs~\cite{Waldron:09,FickusJMPW:18}, block designs~\cite{FickusMT:12,JasperMF:14,FickusJMP:18,FickusJMP:arxiv,FickusJ:arxiv}, and discrete geometry~\cite{FickusMJ:16,FickusJMPW:17}; see~\cite{FickusM:online} for a living survey.

With the exception of~\cite{Creignou:07}, the connections above were known to the authors at the outset of this project.
After completing a draft of the present paper, the authors learned of additional combinatorial objects related to roux.
First, \emph{regular $t$-graphs}~\cite{Higman:90,Kalmanovich:13,Sankey:17} may be seen as ``switching equivalence classes'' of roux for the cyclic group $C_t \leq \mathbb{C}^\times$.
For the specific case $t=3$, Kalmanovich~\cite{Kalmanovich:13} has observed that every regular $3$-graph produces a regular abelian \textsc{drackn}, and this appears to be a specific instance of Theorem~\ref{thm.drackns from roux}.
Next, every roux is an instance of a \emph{regular weight}~\cite{Higman:90}, specifically, one for which the underlying \emph{rainbow} consists of $\{I,J-I\}$, where $J$ is the all-ones matrix.
(The reader may check the details, which follow very quickly from Lemma~\ref{lem.B squared} in the next section).
The relationship between a roux and its roux scheme is similar to an observation by Sankey~\cite{Sankey:17} made in the context of regular weights.
Regular weights, in turn, are specific instances of the more general notion of \emph{gain graphs}~\cite{Zaslavsky:89}, also known as \emph{voltage graphs}~\cite{Gross:74}.

\subsection{Outline}
\label{subsec.outline}

The remainder of this paper is laid out as follows.
Sections~\ref{sec.roux schemes} and~\ref{sec.roux lines} study the association schemes and lines that arise from roux, respectively.
Section~\ref{sec.roux graphs} establishes how roux generalize regular abelian \textsc{drackn}s, and Section~\ref{sec.apps and exs} leverages this generalization to resolve open questions about regular abelian \textsc{drackn}s before providing several examples of doubly transitive lines.
We conclude in Section~\ref{sec.summary} with a summary and the proofs of Theorem~\ref{thm.Higman Pair Theorem}(b) and Corollary~\ref{cor.2tran-roux-etf}.

\section{Roux schemes}
\label{sec.roux schemes}

Given a group $\Gamma$ of order $r$, let $\lceil\cdot\rfloor\colon\mathbb{C}[\Gamma]^{n\times n}\to\mathbb{C}^{rn\times rn}$ denote the injective $*$-algebra homomorphism that applies the Cayley representation of $\Gamma$ (extended linearly to $\mathbb{C}[\Gamma]$) to each entry of the input matrix.
Given a roux $B$ for an abelian group $\Gamma$, then $\{\lceil gI\rfloor\}_{g\in\Gamma}$ and $\{\lceil gB\rfloor\}_{g\in\Gamma}$ form the adjacency matrices of a commutative association scheme whose adjacency algebra is isomorphic to $\mathscr{A}(B)$; we refer to this as the corresponding \textbf{roux scheme}.
(This is akin to the correspondence between regular weights and coherent configurations given in~\cite{Sankey:17}.)
We say a scheme is \textit{roux} if the points have labels in $[n]\times\Gamma$ such that there exists an $n\times n$ roux for $\Gamma$ that produces the scheme through this process.
Roux schemes are important to the study of doubly transitive lines because they provide a combinatorial generalization of Higman pairs:

\begin{theorem}
\label{thm.schurian roux}
Let $G$ be a finite group, and pick $H\leq G$.
The Schurian scheme of $(G,H)$ is isomorphic to a roux scheme if and only if $(G,H)$ is a Higman pair.
\end{theorem}

We will prove Theorem~\ref{thm.schurian roux} later.
The following lemma indicates the origin of the name \textit{roux}: it ``thickens up'' an otherwise thin scheme, imitating the role of roux in the culinary arts.

\begin{lemma}
\label{lem.roux scheme characterization}
An association scheme is isomorphic to a roux scheme if and only if it is commutative and its thin radical acts regularly (by multiplication) on the other adjacency matrices, at least one of which is symmetric.
\end{lemma}

\begin{proof}
($\Rightarrow$)
Suppose $B$ is a roux for $\Gamma$.
Then $\lceil B\rfloor$ is symmetric by (R3).
Furthermore, the thin radical in the roux scheme is $\{\lceil gI\rfloor\}_{g\in\Gamma}$, which acts regularly on the other matrices $\{\lceil gB\rfloor\}_{g\in\Gamma}$ since $\lceil gI\rfloor\lceil hB\rfloor=\lceil ghB\rfloor$ for $g,h\in\Gamma$.
The scheme is commutative since $\Gamma$ is abelian.

($\Leftarrow$)
Let $m$ denote the size of the adjacency matrices, and pick $\Gamma\leq S_m$ so that the thin radical is given by the matrix representations $\{P_\sigma\}_{\sigma\in\Gamma}$.
Select a symmetric adjacency matrix $A$ outside the thin radical.
Since $\Gamma$ acts regularly on these other matrices, they are given by the orbit $\{P_\sigma A\}_{\sigma\in\Gamma}$.
Since the scheme is commutative, $\Gamma$ is abelian.
Next, $\Gamma$ acts on $[m]$, and we claim that the stabilizer of each point is trivial.
To see this, fix $\sigma\in\Gamma$ with $\sigma\neq 1$ and $i\in[m]$.
Then we may pick $j\in[m]$ so that $A_{ij}=1$, in which case $A_{\sigma^{-1}(i),j}=(P_\sigma A)_{ij}=0$, meaning $\sigma^{-1}(i)\neq i$.
As such, the orbits of $[m]$ under $\Gamma$ all have size $r:=|\Gamma|$.
Put $n:=m/r$.

For each orbit $\mathcal{O}_i$, arbitrarily label one of the points $p\in\mathcal{O}_i$ with $\ell(p)=1\in\Gamma$ and label the other points $q\in\mathcal{O}_i$ with $\ell(q)=\sigma\in\Gamma$ such that $\sigma(p)=q$.
Then we may rearrange $[m]$ by ordering the orbits $\mathcal{O}_1,\ldots,\mathcal{O}_n$, and ordering points within each orbit according to $\Gamma$.
Conjugating $\{P_\sigma\}_{\sigma\in\Gamma}\cup\{P_\sigma A\}_{\sigma\in\Gamma}$ by this ordering produces matrices $\{\lceil\sigma I\rfloor\}_{\sigma\in\Gamma}\cup\{\lceil\sigma I\rfloor A'\}_{\sigma\in\Gamma}$. 
Finally, $A'=\lceil\sigma I\rfloor A' \lceil\sigma I\rfloor^{-1}$, and so $A'_{ij}=A'_{\sigma(i),\sigma(j)}$, meaning each $r\times r$ block of $A'$ is $\Gamma$-invariant, i.e., $A'=\lceil B\rfloor$ for some $B\in\mathbb{C}[\Gamma]^{n\times n}$.
By counting, each off-diagonal block of $A'$ is a permutation matrix, and so each off-diagonal entry of $B$ lies in $\Gamma$.
Overall, the adjacency matrices are $\{\lceil\sigma I\rfloor\}_{\sigma\in\Gamma}\cup\{\lceil\sigma B\rfloor\}_{\sigma\in\Gamma}$, and since $B$ satisfies (R1)--(R4), we may conclude that the scheme is roux.
\end{proof}

Next, we offer an alternative to (R4) that is often easier to work with in practice.

\begin{lemma}
\label{lem.B squared}
Suppose $B\in\mathbb{C}[\Gamma]^{n\times n}$ satisfies (R1)--(R3).
Then $B$ is a roux for $\Gamma$ if and only if
\[
B^2=(n-1)I+\sum_{g\in\Gamma}c_ggB
\]
for some complex numbers $\{c_g\}_{g\in\Gamma}$.
In this case, we necessarily have that $\{c_g\}_{g\in\Gamma}$ are nonnegative integers that sum to $n-2$, with $c_{g^{-1}}=c_g$ for every $g\in\Gamma$.
\end{lemma}

\begin{proof}
($\Leftarrow$)
It suffices to demonstrate (R4).
Note that $\{gI\}_{g\in\Gamma}$ and $\{gB\}_{g\in\Gamma}$ commute, and the assumption on $B^2$ gives that their span contains their pairwise products:
\[
(gI)(hI)=ghI,
\qquad
(gI)(hB)=ghB,
\qquad
(gB)(hB)=ghB^2.
\]
Thus, their span is an algebra.

($\Rightarrow$)
The diagonal entries of $B^2$ are given by
\begin{equation}
\label{eq.square's diagonal}
(B^2)_{ii}
=\sum_{j\in[n]}B_{ij}B_{ji}
=\sum_{\substack{j\in[n]\\ j\neq i}} B_{ij}(B_{ij})^{-1}
=n-1.
\end{equation}
As such, the diagonal component of $B^2$ is $(n-1)I$.
The off-diagonal component lies in the span of $\{gB\}_{g\in\Gamma}$.
Thus, we may write
\[
B^2=(n-1)I+\sum_{g\in\Gamma}c_ggB.
\]

For the final claim, we consider an off-diagonal entry of $B^2$:
\[
\sum_{h\in\Gamma}c_{h(B_{ij})^{-1}}h
=\sum_{g\in\Gamma}c_g(gB)_{ij}
=(B^2)_{ij}
=\sum_{k\in[n]} B_{ik}B_{kj}
=\sum_{\substack{k\in[n]\\i\neq k\neq j}} B_{ik}B_{kj}.
\]
The right-hand side is a sum of $n-2$ (not necessarily distinct) members of $\Gamma$, and so we conclude that $\{c_g\}_{g\in\Gamma}$ are nonnegative integers that sum to $n-2$.
Furthermore, $c_{g^{-1}}=c_g$ for every $g\in\Gamma$ since $B^2$ is self-adjoint.
\end{proof}

We will see that $\{c_g\}_{g\in\Gamma}$ serve as fundamental parameters to the study of roux, and we refer to them as \textbf{roux parameters}.
For example, the following result generalizes \eqref{eq.roux example}, but requires a definition:
A \textbf{conference matrix} is an $n\times n$ matrix $M$ with zero diagonal and off-diagonal entries in $\{\pm1\}$ such that $MM^\top=(n-1)I$.
For the lemma statement below, recall that $\delta_\mathrm{i}$ denotes the image of $\mathrm{i}\in C_4$ in the group ring $\mathbb{C}[C_4]$; note the absence of italics to distinguish from an index $i\in [n]$.

\begin{lemma}
\label{lem.conference roux}
Given an $n\times n$ antisymmetric conference matrix $M$, define $B\in\mathbb{C}[C_4]^{n\times n}$ by
\[
B_{ij}=\left\{\begin{array}{ll}
0&\text{if }M_{ij}=0;\\
\delta_{\mathrm{i}}&\text{if }M_{ij}=1;\\
\delta_{-\mathrm{i}}&\text{if }M_{ij}=-1.
\end{array}\right.
\]
Then $B$ is a roux for $C_4$ with parameters $c_{\pm1}=0$ and $c_{\pm \mathrm{i}}=n/2-1$.
\end{lemma}

\begin{proof}
First, (R1)--(R3) hold by definition.
For (R4), we leverage Lemma~\ref{lem.B squared}.
To this end, we first note that the diagonal of $B^2$ satisfies \eqref{eq.square's diagonal}.
Next, take any $i\neq j$.
Then
\[
0
=-(MM^\top)_{ij}
=(M^2)_{ij}
=\sum_{k\in[n]\setminus\{i,j\}}M_{ik}M_{kj},
\]
and since each term on the right-hand side lies in $\{\pm1\}$, we conclude that half of these terms (i.e., $n/2-1$ of them) are $+1$ and the other half are $-1$.
As such,
\[
(B^2)_{ij}
=\sum_{k\in[n]\setminus\{i,j\}}\delta_{\mathrm{i}M_{ik}}\delta_{\mathrm{i}M_{kj}}
=\sum_{k\in[n]\setminus\{i,j\}}\delta_{-M_{ik}M_{kj}}
=\Big(\frac{n}{2}-1\Big)(\delta_1+\delta_{-1})
=\Big(\frac{n}{2}-1\Big)(\delta_{\mathrm{i}}+\delta_{-\mathrm{i}})B_{ij},
\]
implying $B^2=(n-1)I+(n/2-1)\delta_\mathrm{i}B+(n/2-1)\delta_{-\mathrm{i}}B$, as desired.
\end{proof}

\begin{lemma}[Basic roux transformations]
\label{lem.basic roux trans}
Take any $n\times n$ roux $B$ with parameters $\{c_g\}_{g\in \Gamma}$.
\begin{itemize}
\item[(a)]
Given a diagonal matrix $D\in\mathbb{C}[\Gamma]^{n\times n}$ with $D_{ii}\in\Gamma$ for every $i\in[n]$, then $DBD^{-1}$ is a roux for $\Gamma$ with parameters $\{c_g\}_{g\in \Gamma}$.
\item[(b)]
Given $h\in\Gamma$, then $hB$ is a roux for $\Gamma$ if and only if $h^2=1$.
In that case, its roux parameters are $\{c_{gh}\}_{g\in\Gamma}$.
\item[(c)]
Given a homomorphism $\varphi\colon\Gamma\to\Lambda$, extend $\varphi$ linearly to the group ring and apply entrywise to get $\bar\varphi\colon\mathbb{C}[\Gamma]^{n\times n}\to\mathbb{C}[\Lambda]^{n\times n}$.
Then $\bar\varphi(B)$ is a roux for $\Lambda$ with parameters $\big\{\sum_{g\in\varphi^{-1}(\lambda)}c_g\big\}_{\lambda\in\Lambda}$, where the empty sum is taken to be zero.
\item[(d)]
Given a group $\Lambda\geq\Gamma$, then $B$ is a roux for $\Lambda$ with parameters $\{c'_\lambda\}_{\lambda\in\Lambda}$, where $c'_\lambda=c_\lambda$ if $\lambda\in\Gamma$ and $c'_\lambda=0$ otherwise.
\end{itemize}
\end{lemma}

These roux transformations suggest various invariants.
Part (a) establishes ``switching equivalence classes'' of roux.
Notice that a reasonable representative of each class takes $B_{i,1}$ and $B_{1,i}$ to be the identity element of $\Gamma$ for every $i\neq 1$.
For (b), we note that $\lceil B\rfloor$ and $\lceil hB\rfloor$ are adjacency matrices in the same roux scheme, meaning each roux generates the same scheme, though with relations re-indexed.
Given a roux scheme with a fixed indexing of the vertices, the $hB$'s in part (b) are the only roux that produce this scheme.
While (d) explains how to view $B$ as a roux for a supergroup, Lemma~\ref{lem.inverse prob} in the next section shows how to view $B$ as a roux for a subgroup (provided the roux parameters are zero on the complement of the subgroup).

\begin{proof}[Proof of Lemma~\ref{lem.basic roux trans}]
First (a) and (c) are straightforward, as is ($\Leftarrow$) in (b).
For ($\Rightarrow$), define $A=hB$ and pick $i$ and $j$ such that $i\neq j$.
Then (R3) implies
\[
hB_{ij}
=A_{ij}
=(A_{ji})^{-1}
=(hB_{ji})^{-1}
=h^{-1}(B_{ji})^{-1}
=h^{-1}B_{ij}.
\]
Multiplying both sides by $h(B_{ij})^{-1}$ then gives $h^2=1$.
Finally, (d) follows from (c) since the natural injection $\Gamma\to\Lambda$ is a homomorphism.
\end{proof}

In order to prove ($\Leftarrow$) in Theorem~\ref{thm.schurian roux}, we need a technical lemma regarding the structure of Higman pairs:

\begin{lemma}
\label{lem.basic higman pair facts}
Given a Higman pair $(G,H)$, denote $K=N_G(H)$, $n=[G:K]$ and $r=[K:H]$, and select any key $b\in G\setminus K$.
Then
\begin{itemize}
\item[(a)]
$H$ has $2r$ double cosets in $G$: $r$ of the form $aH$, and $r$ of the form $HabH$ for some $a\in K$;
\item[(b)]
for every $a\in K$, we have $HabH=HbaH$; and
\item[(c)]
for every $a\in K$, we have $|HabH|=(n-1)|H|$.
\end{itemize}
\end{lemma}

\begin{proof}
For (a), (H1) implies that $G$ is a disjoint union of $K$ and $KbK$.
Next, $K$ is covered by left cosets of $H$ in $K$ (these are double cosets of $H$ in $G$ since $K$ normalizes $H$), while $KbK$ is covered by sets of the form
\[
(aH)b(a'H)
=aa'(a')^{-1}Hba'H
=aa'H(a')^{-1}ba'H
\stackrel{(*)}{\subseteq}aa'HbH
=Haa'bH
\qquad
(a,a'\in K);
\]
here, $(*)$ applies (H4).
Since $G$ can be partitioned into double cosets of $H$, we conclude that every double coset of $H$ has the form $aH$ or $HabH$ for some $a\in K$.
To count these double cosets, consider the action of $K/H$ on the double cosets defined by $aH\cdot HxH = HaxH$.
There are two orbits under this action: those of the form $aH$, and those of the form $HabH$.
In particular, $H$ is the stabilizer of $H$, and (H5) gives that $H$ is also the stabilizer of $HbH$, meaning $K/H$ acts regularly on both orbits.
This gives (a).

Next, we apply (H4) to get
\[
HbaH
=Haa^{-1}baH
=aHa^{-1}baH
\subseteq aHbH
=HabH.
\]
We obtain equality by counting: $|HbaH|=|HbHa|=|HbH|=|aHbH|=|HabH|$.
This gives (b).

For (c), note that our proof of (b) gives that $|HabH|=|HbH|$ for every $a\in K$.
It suffices to show that $|HbH|=(n-1)|H|$.
Recall that the double cosets of the form $aH$ cover $K$, whereas the double cosets of the form $HabH$ cover $KbK=G\setminus K$.
By (a), we therefore have
\[
r|HbH|
=|KbK|
=|G|-|K|
=(n-1)|K|
=r(n-1)|H|,
\]
and division by $r$ gives the result.
\end{proof}

\begin{lemma}[Roux from Higman pairs]
\label{lem.higman's roux}
Given a Higman pair $(G,H)$, denote $K=N_G(H)$ and $n=[G:K]$, and select any key $b\in G\setminus K$.
Choose left coset representatives $\{x_j\}_{j\in[n]}$ for $K$ in $G$ and choose coset representatives $\{a_g\}_{g\in K/H}$ for $H$ in $K$.
Define $B\in\mathbb{C}[K/H]^{n\times n}$ entrywise as follows:
Given $i\neq j$, let $B_{ij}$ be the unique $g\in K/H$ for which $x_i^{-1}x_j\in Ha_gbH$, and set $B_{ii}=0$.
Then $B$ is a roux for $K/H$ with roux parameters $\{c_g\}_{g\in K/H}$ given by
\[
c_g
=\frac{n-1}{|H|}\cdot|bHb^{-1}\cap Ha_gbH|.
\]
Furthermore, the roux scheme generated by $B$ is isomorphic to the Schurian scheme of $(G,H)$.
\end{lemma}

Notice that a different choice of $x_j$'s produces a switching equivalent roux (as in Lemma~\ref{lem.basic roux trans}(a)), whereas a different choice of $a_g$'s makes no change to $B$.

\begin{proof}[Proof of Lemma~\ref{lem.higman's roux}]
We start by checking (R1)--(R4).
First, $B$ satisfies (R1) and (R2) by definition.
For (R3), pick $i\neq j$ and put $g=B_{ij}$.
Then
\[
x_j^{-1}x_i
=(x_i^{-1}x_j)^{-1}
\in(a_{g}HbH)^{-1}
=Hb^{-1}Ha_{g^{-1}}
\stackrel{(*)}{=}HbHa_{g^{-1}}
=Hba_{g^{-1}}H
\stackrel{(\dagger)}{=}Ha_{g^{-1}}bH,
\]
where $(*)$ applies (H3) and $(\dagger)$ follows from Lemma~\ref{lem.basic higman pair facts}(b).
As such, we have $B_{ji}=g^{-1}$, implying (R3).
To verify (R4), first observe that $[n] \times K/H$ is in bijection with $G/H$ through the mapping $(i,h) \mapsto x_i a_h H$.
Let $\psi\colon \mathbb{C}^{([n]\times K/H)\times ([n] \times K/H)} \to \mathbb{C}^{G/H \times G/H}$ be the corresponding $*$-algebra isomorphism
\[
\psi(M)_{x_ia_hH,x_ja_kH}
= M_{(i,h),(j,k)}
\qquad
(M\in \mathbb{C}^{([n]\times K/H)\times ([n] \times K/H)},\, i,j \in [n],\, h,k \in K/H).
\]
(Here, we identify $\mathbb{C}^{([n]\times K/H)\times ([n] \times K/H)}$ with $\mathbb{C}^{n|K/H|\times n|K/H|}$ in the usual way using the ordering on $K/H$ induced by our indexed coset representatives.)
Pre-composing with $\lceil\cdot\rfloor\colon\mathbb{C}[K/H]^{n\times n}\to\mathbb{C}^{([n]\times K/H) \times ([n] \times K/H)}$ gives an injective $*$-algebra homomorphism of $\mathbb{C}[K/H]^{n\times n}$ into $\mathbb{C}^{G/H\times G/H}$.
Recall the notation from \eqref{eq.schurian adjacency matrices}.
It is straightforward (but annoying) to verify that $\psi(\lceil gI \rfloor) = A_{a_gH}$ and $\psi(\lceil gB \rfloor) = A_{Ha_gbH}$ for all $g \in K/H$.
Considering Lemma~\ref{lem.basic higman pair facts}(a), these images span a $*$-subalgebra $\mathscr{A}(G,H)$ of $\mathbb{C}^{G/H\times G/H}$, and so $\{gI\}_{g\in K/H}$ and $\{gB\}_{g \in K/H}$ span a $*$-algebra in $\mathbb{C}[K/H]^{n\times n}$.
This is (R4).
We conclude that $B$ is a roux, and that the scheme it generates is isomorphic to the Schurian scheme of $(G,H)$.

It remains to compute the roux parameters $\{c_g\}_{g\in K/H}$.
We accomplish this by computing $B^2$.
To this end, we denote $\iota \colon \mathscr{A}(B) \to \mathscr{A}(G,H)$ for the $*$-algebra isomorphism
\[
\iota(gI)
= A_{a_gH},
\qquad
\iota(gB) = A_{Ha_gbH}
\qquad
(g \in K/H),
\]
as above.
Recalling \eqref{eq.isomorphisms}, it is convenient to perform much of this computation in an anti-isomorphic domain:
\[
(\theta\circ\phi\circ\iota)(B)=\frac{1}{|H|}\sum_{x\in HbH}x\in \mathbb{C}[H\backslash G/H].
\]
We have
\[
(\theta\circ\phi\circ\iota)(B^2)
=[(\theta\circ\phi\circ\iota)(B)]^2
=\frac{1}{|H|^2}\sum_{x_1\in HbH}x_1\sum_{x_2\in HbH}x_2
=\frac{1}{|H|^2}\sum_{x_1\in HbH}x_1\sum_{x_2\in Hb^{-1}H}x_2,
\]
where the last step follows from (H3).
Next, consider the action of $H\times H$ on $G$ defined by $(h_1,h_2)\cdot x=h_1xh_2^{-1}$.
The orbits of this action are the double cosets of $H$ in $G$.
As such, we may continue with the help of the orbit-stabilizer theorem:
\[
(\theta\circ\phi\circ\iota)(B^2)
=\frac{|HbH|^2}{|H|^6}\sum_{h_1,h_2\in H}h_1bh_2\sum_{h_3,h_4\in H}h_3b^{-1}h_4
=\frac{|HbH|^2}{|H|^5}\sum_{h_1,h_2,h_3\in H}h_1bh_2b^{-1}h_3,
\]
where the last step changes variables $(h_2h_3,h_4)\mapsto(h_2,h_3)$.
At this point, we observe that $h_1bh_2b^{-1}h_3=x$ precisely when $h_1^{-1}xh_3^{-1}=bh_2b^{-1}$, and so
\begin{align*}
(\theta\circ\phi\circ\iota)(B^2)
&=\frac{|HbH|^2}{|H|^5}\sum_{x\in G}|\{(h_1,h_2,h_3)\in H^3:h_1bh_2b^{-1}h_3=x\}|x\\
&=\frac{|HbH|^2}{|H|^5}\sum_{x\in G}|\{(h_1,h_3)\in H^2:h_1^{-1}xh_3^{-1}\in bHb^{-1}\}|x\\
&=\frac{|HbH|^2}{|H|^5}\sum_{x\in G}\frac{|H|^2}{|HxH|}\cdot|bHb^{-1}\cap HxH|x,
\end{align*}
where the normalization in the last step follows from the orbit-stabilizer theorem, as before.
We now apply $(\theta\circ\phi\circ\iota)^{-1}$ to both sides to get
\begin{align*}
B^2
&=\frac{|HbH|^2}{|H|^4}\bigg(\sum_{g\in K/H}\frac{|H|^2}{|a_gH|}\cdot|bHb^{-1}\cap a_gH|gI+\sum_{g\in K/H}\frac{|H|^2}{|Ha_gbH|}\cdot|bHb^{-1}\cap Ha_gbH|gB\bigg)\\
&=(n-1)I+\sum_{g\in K/H}\frac{n-1}{|H|}\cdot|bHb^{-1}\cap Ha_gbH|gB,
\end{align*}
where the final simplification follows from \eqref{eq.square's diagonal} and Lemma~\ref{lem.basic higman pair facts}(c).
\end{proof}

\begin{proof}[Proof of Theorem~\ref{thm.schurian roux}]
It suffices to prove ($\Rightarrow$), since ($\Leftarrow$) follows immediately from Lemma~\ref{lem.higman's roux}.
Suppose the Schurian scheme of $(G,H)$ is isomorphic to a roux scheme.
By Lemma~\ref{lem.roux scheme characterization}, we equivalently have that the scheme is commutative and its thin radical acts regularly on the other adjacency matrices, at least one of which is symmetric.

First, it is straightforward to show that the thin radical is comprised of the adjacency matrices $\{A_{aH}\}_{aH\in K/H}$ defined in \eqref{eq.schurian adjacency matrices}, where $K=N_G(H)$.
Before proceeding, we make a general observation:
\begin{equation}
\label{eq.general observation}
A_{aH}A_{HxH}=A_{HaxH}
\qquad
\text{for every}
\qquad
a\in K,~x\in G.
\end{equation}
To see this, note that $\phi(A_{aH})=\frac{1}{|H|}\cdot\mathbf{1}_{aH}$ and $\phi(A_{HxH})=\frac{1}{|H|}\cdot\mathbf{1}_{HxH}$.
Applying the commutativity of our scheme and the fact that $\phi$ is an anti-isomorphism, we have
\[
(\phi(A_{aH}A_{HxH}))(y)
=(\phi(A_{HxH}A_{aH}))(y)
=\frac{1}{|H|^2}(\mathbf{1}_{aH}*\mathbf{1}_{HxH})(y)
=\frac{1}{|H|^2} \sum_{z\in G} \mathbf{1}_{aH}(z) \mathbf{1}_{HxH}(z^{-1} y).
\]
Next, change variables $u = a^{-1} z$ and observe that $au \in aH$ if and only if $u \in H$.
Therefore,
\[
(\phi(A_{aH}A_{HxH}))(y)
=\frac{1}{|H|^2} \sum_{u\in G} \mathbf{1}_{aH}(au) \mathbf{1}_{HxH}(u^{-1} a^{-1} y)
=\frac{1}{|H|^2} \sum_{u\in H} \mathbf{1}_{HxH}(u^{-1} a^{-1} y).
\]
Finally, for $u \in H$ we have $u^{-1} a^{-1} y \in HxH$ if and only if $a^{-1} y \in HxH$, if and only if $y \in aHxH = HaxH$.
As such,
\[
(\phi(A_{aH}A_{HxH}))(y)
=\frac{1}{|H|^2} \sum_{u\in H} \mathbf{1}_{HaxH}(y)
= \frac{1}{|H|} \mathbf{1}_{HaxH}(y)
=(\phi(A_{HaxH}))(y).
\]

By considering the case $x\in K$ in \eqref{eq.general observation}, we have that the thin radical is isomorphic to $K/H$.
Since the scheme is commutative by assumption, this group is abelian, and so we have (H2).
We also have that one of the adjacency matrices $A_{HbH}$ is symmetric by assumption, which is equivalent to $HbH=Hb^{-1}H$, namely, (H3).
Next, take any $a\in K$.
Then \eqref{eq.general observation} and commutativity together give
\[
A_{HabH}
=A_{aH}A_{HbH}
=A_{HbH}A_{aH}
=(A_{a^{-1}H}A_{Hb^{-1}H})^\top
=A_{Ha^{-1}b^{-1}H}^\top
=A_{HbaH}.
\]
As such, $HabH=HbaH$, with which we establish
\[
aba^{-1}
\in Haba^{-1}H
=HabHa^{-1}
=HbaHa^{-1}
=HbH,
\]
i.e., (H4).
For (H5), take any $a\in K$ such that $ab\in HbH$.
Then \eqref{eq.general observation} gives
\[
A_{aH}A_{HbH}
=A_{HabH}
=A_{HbH}.
\]
Recall that by assumption, the thin radical acts regularly on the other adjacency matrices, meaning the stabilizer of $A_{HbH}$ under this action is trivial.
As such, we have $aH=H$, meaning $a\in H$.
Finally, for (H1), take any $x\in G$ and suppose for the moment that $x\not\in K$.
Then $A_{HxH}$ is not thin, and so there exists $a\in K$ such that $A_{aH}A_{HbH}=A_{HxH}$ since the thin radical acts transitively on the non-thin adjacency matrices.
By \eqref{eq.general observation}, this in turn implies $HxH=HabH\subseteq KbH\subseteq KbK$.
Overall, $x\in G$ either belongs to $K$ or $KbK$, meaning $K$ has two double cosets, and therefore $G$ acts doubly transitively on $G/K$.
\end{proof}

Given a commutative association scheme, the primitive idempotents provide an alternative orthogonal basis for the adjacency algebra.
As detailed in Section~\ref{sec.intro}, this basis is particularly important to our pursuit of Gram matrices.
In the case of roux schemes, all of the primitive idempotents can be expressed in terms of the characters of $\Gamma$ and the underlying roux.
Denoting $\hat{\Gamma}$ for the Pontryagin dual group of characters of $\Gamma$, we have the following:

\begin{theorem}
\label{thm.roux primitive idempotents}
Given an $n\times n$ roux $B$ for $\Gamma$ with parameters $\{ c_g \}_{g \in \Gamma}$, the primitive idempotents for the corresponding roux scheme are scalar multiples of
\[
\mathcal{G}_{\alpha}^\epsilon
:=\sum_{g\in\Gamma}\alpha(g)\lceil gI\rfloor+\mu_\alpha^\epsilon\sum_{g\in\Gamma}\alpha(g) \lceil gB \rfloor,
\qquad
\qquad
(\alpha\in\hat\Gamma,~\epsilon\in\{+,-\}),
\]
where $\mu_\alpha^\epsilon$ is defined in terms of the Fourier transform $\hat{c}_\alpha:=\sum_{h\in\Gamma}c_h\overline{\alpha(h)}$ as follows:
\[
\mu_\alpha^\epsilon
=\frac{\hat{c}_\alpha+\epsilon\sqrt{(\hat{c}_\alpha)^2+4(n-1)}}{2(n-1)}.
\]
Furthermore,
\[
d_\alpha^\epsilon
:=\operatorname{rank}(\mathcal{G}_{\alpha}^\epsilon)
=\frac{n}{1+(n-1)(\mu_\alpha^\epsilon)^2}.
\]

Moreover, if $\alpha \in \hat{\Gamma}$, $\mu > 0$, and 
\[ \mathcal{G} := \sum_{g\in\Gamma}\alpha(g)\lceil gI\rfloor+\mu \sum_{g\in\Gamma}\alpha(g) \lceil gB \rfloor \]
satisfies $\mathcal{G}^2 = c \mathcal{G}$ for some $c >0$, then $\mu = \mu_\alpha^\epsilon$ for some $\epsilon \in \{ +, - \}$, and a scalar multiple of $\mathcal{G} = \mathcal{G}_\alpha^\epsilon$ is a primitive idempotent for the roux scheme.
\end{theorem}

Notice that $\mathcal{G}_\alpha^\epsilon$ is the Gram matrix of $|\Gamma|$ phased versions of all $n$ vectors of an ETF in $\mathbb{C}^{d_\alpha^\epsilon}$ with coherence $|\mu_\alpha^\epsilon|$.
Expanding $d_\alpha^\epsilon$ in terms of the definition of $\mu_\alpha^\epsilon$ gives
\[
d_\alpha^\epsilon
=\frac{2n(n-1)}{(\hat{c}_\alpha)^2+4(n-1)+\epsilon\hat{c}_\alpha\sqrt{(\hat{c}_\alpha)^2+4(n-1)}}.
\]
By appearances, it seems that $d_\alpha^\epsilon\in\mathbb{Z}$ is a strong necessary condition for the existence of roux.

\begin{proof}[Proof of Theorem~\ref{thm.roux primitive idempotents}]
First, we establish that $\mathcal{G}$ is a scalar multiple of an idempotent if and only if $\mu = \mu_\alpha^\epsilon$.
To this end, it is helpful to write
\[
\mathcal{G}
=\bigg(\sum_{g\in\Gamma}\alpha(g)\lceil gI\rfloor\bigg)\bigg(\lceil I\rfloor+\mu\lceil B \rfloor\bigg)
=:M_1M_2.
\]
If we put $r=|\Gamma|$, then
\[
M_1^2
=r\sum_{g\in\Gamma}\alpha(g)\lceil gI\rfloor,
\qquad
M_2^2
=\Big(1+(n-1)\mu^2\Big)\lceil I\rfloor+2\mu\lceil B\rfloor+\mu^2\bigg(\sum_{g\in\Gamma} c_g\lceil gI\rfloor\bigg)\lceil B\rfloor.
\]
Indeed, $M_1^2$ is computed by a change of variables, whereas $M_2^2$ is computed by applying the formula for $B^2$ in terms of the roux parameters.
With this, we may compute $\mathcal{G}^2=M_1^2M_2^2$, the third term of which is $r\mu^2M_3\lceil B\rfloor$, where
\[
M_3
=\bigg(\sum_{g\in\Gamma}\alpha(g)\lceil gI\rfloor\bigg)\bigg(\sum_{g\in\Gamma} c_g\lceil gI\rfloor\bigg)
=\sum_{g,h\in\Gamma} c_g\alpha(h)\lceil ghI\rfloor
=\sum_{g,h\in\Gamma} c_h\alpha(gh^{-1})\lceil gI\rfloor
=\hat{c}_\alpha\sum_{g\in\Gamma}\alpha(g)\lceil gI\rfloor.
\]
Putting everything together, we have
\[
\mathcal{G}^2
=r\Big(1+(n-1)\mu^2\Big)\cdot\sum_{g\in\Gamma}\alpha(g)\lceil gI\rfloor+r\Big(2+\hat{c}_\alpha\mu\Big)\cdot\mu\sum_{g\in \Gamma}\alpha(g)\lceil gB\rfloor.
\]
By linear independence, $\mathcal{G}^2$ is a scalar multiple of $\mathcal{G}$ if and only if $1+(n-1)\mu^2=2+\hat{c}_\alpha\mu$, if and only if $\mu = \mu_\alpha^\epsilon$.
Notice that $\mu_\alpha^\epsilon$ is real since $c_{g^{-1}}=c_g$ for every $g\in\Gamma$ (see Lemma~\ref{lem.B squared}), and $\lceil B \rfloor$ is symmetric since $B$ is self-adjoint, and so $M_2^*=M_2^\top=M_2$.
Also, a change of variables gives that $M_1^*=M_1$, and so $(\mathcal{G}_{\alpha}^\epsilon)^*=\mathcal{G}_{\alpha}^\epsilon$.
Hence, $\mathcal{G}_{\alpha}^\epsilon$ is a scalar multiple of an orthogonal projection matrix.

Next, we show that $\mathcal{G}_{\alpha}^\epsilon$ is a scalar multiple of a \textit{primitive} idempotent.
Since the dimension of $\mathscr{A}(B)$ is $2r$ and there are $2r$ different $\mathcal{G}_{\alpha}^\epsilon$'s, it suffices to show that $\{\mathcal{G}_{\alpha}^\epsilon\}$ are mutually orthogonal.
To this end, take $\alpha,\beta\in\hat\Gamma$ and $\epsilon,\delta\in\{+,-\}$.
We will proceed in two cases.
First, suppose $\alpha\neq\beta$.
Then
\begin{equation}
\label{eq.orthog helper}
\sum_{g,h\in\Gamma}\alpha(g)\beta(h)\lceil gI\rfloor\lceil hI\rfloor
=\sum_{g,k\in\Gamma}\alpha(g)\beta(g^{-1}k)\lceil kI\rfloor
=\bigg(\sum_{g\in\Gamma}\alpha(g)\overline{\beta(g)}\bigg)\bigg(\sum_{k\in\Gamma}\beta(k)\lceil kI\rfloor\bigg)
=0,
\end{equation}
where the last step is by the orthogonality of characters.
As such,
\[
\mathcal{G}_{\alpha}^\epsilon\mathcal{G}_{\beta}^\delta
=\bigg(\sum_{g\in\Gamma}\alpha(g)\lceil gI\rfloor\bigg)\bigg(\lceil I\rfloor+\mu_\alpha^\epsilon\lceil B \rfloor\bigg)\bigg(\sum_{h\in\Gamma}\beta(h)\lceil hI\rfloor\bigg)\bigg(\lceil I\rfloor+\mu_\beta^\delta\lceil B \rfloor\bigg)
=0,
\]
where the last step follows from exploiting commutativity to multiply the first and third factors and then applying \eqref{eq.orthog helper}.
This completes the first case.
It remains to show that $\mathcal{G}_{\alpha}^+\mathcal{G}_{\alpha}^-=0$ for every $\alpha\in\hat\Gamma$.
To this end, fix $\alpha\in\hat\Gamma$ and note that
\[
\mu_\alpha^++\mu_\alpha^-
=\frac{\hat{c}_\alpha}{n-1},
\qquad
\mu_\alpha^+\mu_\alpha^-
=-\frac{1}{n-1}.
\]
Combining this with the expression for $B^2$ then gives
\[
\bigg(\lceil I\rfloor+\mu_\alpha^+\lceil B \rfloor\bigg)\bigg(\lceil I\rfloor+\mu_\alpha^-\lceil B \rfloor\bigg)
=\lceil I\rfloor+\frac{\hat{c}_\alpha}{n-1}\lceil B \rfloor-\frac{1}{n-1}\lceil B \rfloor^2
=\frac{1}{n-1}\lceil B\rfloor\bigg(\hat{c}_\alpha\lceil I\rfloor-\sum_{g\in\Gamma}c_g\lceil gI\rfloor\bigg).
\]
With this, we compute the desired product:
\begin{align*}
\mathcal{G}_{\alpha}^+\mathcal{G}_{\alpha}^-
&=\frac{1}{n-1}\lceil B\rfloor\bigg(\sum_{g\in\Gamma}\alpha(g)\lceil gI\rfloor\bigg)^2\bigg(\hat{c}_\alpha\lceil I\rfloor-\sum_{g\in\Gamma}c_g\lceil gI\rfloor\bigg)\\
&=\frac{1}{n-1}\lceil B\rfloor\bigg(\sum_{g\in\Gamma}\alpha(g)\lceil gI\rfloor\bigg)\bigg(\sum_{g,h\in\Gamma}\alpha(gh^{-1})c_h\lceil gI\rfloor-\sum_{g,h\in\Gamma}\alpha(g)c_h\lceil ghI\rfloor\bigg)
=0,
\end{align*}
where the last step follows from a change of variables.

At this point, we have that the primitive idempotents of $\mathscr{A}(B)$ are given by
\[
\frac{1}{r(1+(n-1)(\mu_\alpha^\epsilon)^2)}\cdot\mathcal{G}_\alpha^\epsilon,
\qquad
(\alpha\in\hat\Gamma,~\epsilon\in\{+,-\}).
\]
For the last claim, we need to compute the ranks of these idempotents, which amounts to a trace calculation.
To this end, we isolate the diagonal contribution to $\mathcal{G}_\alpha^\epsilon$ to get $\operatorname{tr}(\mathcal{G}_\alpha^\epsilon)=\operatorname{tr}(\lceil I\rfloor)=rn$, from which the formula for $d_\alpha^\epsilon$ follows.
\end{proof}

\section{Roux lines}
\label{sec.roux lines}

Given unit-norm representatives of linearly dependent equiangular lines, the Gram matrix of these vectors has the form $I+\mu \mathcal{S}$ for some $\mu>0$ and signature matrix $\mathcal{S}$.
As discussed in Subsection~\ref{subsec.example main results}, we call the equiangular lines \textit{roux} if there exist unit-norm representatives such that the signature matrix $\mathcal{S}$ can be obtained by evaluating some roux at a character.
This operation of evaluating at a character requires some notation:
Every $\alpha\in\hat\Gamma$ extends to a $*$-algebra homomorphism $\tilde\alpha\colon\mathbb{C}[\Gamma]\to\mathbb{C}$, which in turn extends to a $*$-algebra homomorphism $\hat\alpha\colon\mathbb{C}[\Gamma]^{n\times n}\to\mathbb{C}^{n\times n}$ given by applying $\tilde\alpha$ entrywise.

\begin{theorem}
\label{thm.roux signature}
Suppose $B\in\mathbb{C}[\Gamma]^{n\times n}$ satisfies (R1)--(R3).
Then $B$ is a roux if and only if for every $\alpha\in\hat{\Gamma}$, $\hat\alpha(B)$ is the signature matrix of an equiangular tight frame.
\end{theorem}

See Theorem~\ref{thm.drackn vs roux} for the combinatorial significance of (R1)--(R3).

\begin{proof}[Proof of Theorem~\ref{thm.roux signature}]
($\Rightarrow$)
Put $\mathcal{S}=\hat\alpha(B)$.
Then $\mathcal{S}=\mathcal{S}^*$ and Lemma~\ref{lem.B squared} gives
\begin{equation}
\label{eq.roux signature matrix}
\mathcal{S}^2
=(n-1)I+\bigg(\sum_{g\in\Gamma}c_g\alpha(g)\bigg)\mathcal{S}.
\end{equation}
As such, $\mathcal{S}$ has at most two eigenvalues $\lambda_1\geq\lambda_2$, and since $\operatorname{tr}(\mathcal{S})=0$, we have $\lambda_1>0>\lambda_2$.
By the spectral theorem, $(\mathcal{S}-\lambda_2I)/(\lambda_1-\lambda_2)$ is an orthogonal projection with off-diagonal of constant modulus.

($\Leftarrow$)
By Lemma~\ref{lem.B squared}, it suffices to compute $\lceil B\rfloor^2$.
Since
\[
\lceil B_{ij}\rfloor_{g,h}
=\left\{\begin{array}{ll}1&\text{if }B_{ij}h=g\\0&\text{otherwise}\end{array}\right\}
=\frac{1}{|\hat\Gamma|}\sum_{\alpha\in\hat\Gamma}\alpha(B_{ij}hg^{-1}),
\]
we may decompose $\lceil B\rfloor$ as follows:
\[
\lceil B\rfloor=\sum_{\alpha\in\hat\Gamma}\hat\alpha(B)\otimes v_\alpha v_\alpha^*,
\qquad
(v_\alpha)_g:=\frac{1}{\sqrt{|\Gamma|}}\cdot\overline{\alpha(g)}.
\]
This decomposition provides a useful expression for $\lceil B\rfloor^2$:
\begin{equation}
\label{eq.etf to roux 1}
\lceil B\rfloor^2
=\sum_{\alpha,\beta\in\hat\Gamma}\hat\alpha(B)\hat\beta(B)\otimes v_\alpha v_\alpha^*v_\beta v_\beta^*
=\sum_{\alpha\in\hat\Gamma}(\hat\alpha(B))^2\otimes v_\alpha v_\alpha^*.
\end{equation}
Next, since $\hat\alpha(B)$ is the signature matrix of an ETF by assumption, it necessarily has exactly two eigenvalues.
Furthermore, (R1)--(R3) together imply that every diagonal entry of $(\hat\alpha(B))^2$ is $n-1$, and so we may write
\begin{equation}
\label{eq.etf to roux 2}
(\hat\alpha(B))^2
=(n-1)I+C_\alpha\cdot\hat\alpha(B),
\qquad
(\alpha\in\hat\Gamma),
\end{equation}
for some sequence $\{C_\alpha\}_{\alpha\in\hat\Gamma}$ in $\mathbb{C}$.
Consider the sequence $\{c_g\}_{g\in\Gamma}$ whose Fourier transform is given by $\hat{c}_\alpha=C_{\alpha^{-1}}$ for $\alpha\in\hat\Gamma$.
We combine this with \eqref{eq.etf to roux 2} to continue \eqref{eq.etf to roux 1}:
\begin{equation}
\label{eq.etf to roux 3}
\lceil B\rfloor^2
=(n-1)\lceil I\rfloor+\sum_{\alpha\in\hat\Gamma}\hat{c}_{\alpha^{-1}}\cdot\hat\alpha(B)\otimes v_\alpha v_\alpha^*
=:(n-1)\lceil I\rfloor+M,
\end{equation}
where the first term follows from the fact that $\sum_{\alpha\in\hat\Gamma}v_\alpha v_\alpha^*=I$.
By (R1), we have that $M_{(i,g),(j,h)}=0$ whenever $i=j$.
For $i\neq j$, we have
\[
M_{(i,g),(j,h)}
=\sum_{\alpha\in\hat\Gamma}\hat{c}_{\alpha^{-1}}\cdot\alpha(B_{ij})(v_\alpha)_g \overline{(v_\alpha)_h}
=\frac{1}{|\hat\Gamma|}\sum_{\alpha\in\hat\Gamma}\hat{c}_\alpha\alpha(B_{ij}^{-1}gh^{-1})
=c_{B_{ij}^{-1}gh^{-1}},
\]
which matches the desired sum:
\[
\bigg(\sum_{k\in\Gamma}c_k\lceil kB\rfloor\bigg)_{(i,g),(j,h)}
=\sum_{k\in\Gamma}c_k\lceil kB_{ij}\rfloor_{g,h}
=\sum_{k\in\Gamma}c_k\left\{\begin{array}{ll}1&\text{if }kB_{ij}h=g\\0&\text{otherwise}\end{array}\right\}
=c_{B_{ij}^{-1}gh^{-1}}.
\]
Overall, $M=\sum_{g\in\Gamma}c_g\lceil gB\rfloor$, and so \eqref{eq.etf to roux 3} and Lemma~\ref{lem.B squared} together give that $B$ is a roux with parameters $\{c_g\}_{g\in\Gamma}$.
\end{proof}

Recalling the primitive idempotents in Theorem~\ref{thm.roux primitive idempotents}, we see that $(\mathcal{G}_{\alpha^{-1}}^+)_{(i,1),(j,1)}=\mu_{\alpha^{-1}}^+\cdot\alpha(B_{ij})$ whenever $i\neq j$.
This implies a fundamental relationship: 

\begin{lemma}
\label{lem.two reps of same lines}
For each $\alpha\in\hat\Gamma$, the signature matrix $\hat{\alpha}(B)$ from Theorem~\ref{thm.roux signature} and the Gram matrix $\mathcal{G}_{\alpha^{-1}}^+$ from Theorem~\ref{thm.roux primitive idempotents} describe the same lines (each line implicated by the former is represented $|\Gamma|$ times in the latter).
\end{lemma}

In fact, this relationship can be used to characterize roux lines (see Corollary~\ref{cor.roux lines detector} for a nicer characterization):

\begin{theorem}
\label{thm.grampa's roux}
Let $\mathscr{L}$ be a sequence of linearly dependent complex lines.
Then $\mathscr{L}$ is roux if and only if all of the following occur simultaneously:
\begin{itemize}
\item[(a)]
$\mathscr{L}$ is equiangular,
\item[(b)]
there exist unit-norm representatives $\{\varphi_i\}_{i\in[n]}$ of $\mathscr{L}$ whose signature matrix is comprised of $r$th roots of unity for some $r$, and
\item[(c)]
the Gram matrix $\mathcal{G} \in \mathbb{C}^{n\times n}$ of $\{g\varphi_i\}_{i\in[n],g\in C_r \subseteq \mathbb{C}}$ carries an association scheme.
\end{itemize}
In this case, there is a roux $B$ such that $\mathscr{L}$ are roux lines for $B$, the association scheme carried by $\mathcal{G}$ is the roux scheme of $B$, and a scalar multiple of $\mathcal{G}$ is a primitive idempotent for that scheme.
Moreover, $\{\varphi_i\}_{i\in[n]}$ is an equiangular tight frame for its span.
\end{theorem}

In the sequel paper~\cite{IversonM:future}, we will see that doubly transitive lines are characterized by (a)--(c), with the added condition that the association scheme of (c) is Schurian.
Taken together, these two results help to clarify that roux lines are ``non-Schurian'' analogues of doubly transitive lines.

\begin{proof}[Proof of Theorem~\ref{thm.grampa's roux}]
We start with a general observation.
Suppose the Gram matrix of $\{\varphi_i\}_{i\in[n]}$ has the form $I+\mu\mathcal{S}$, where $\mu>0$ and $\mathcal{S}$ has entries in the cyclic group $C_r \leq \mathbb{C}^\times$, and define $\tilde{B}\in\mathbb{C}[C_r]^{n\times n}$ to have entries $\tilde{B}_{ii}=0$ and $\tilde{B}_{ij}=\delta_{\mathcal{S}_{ij}}$ for $i\neq j$.
(As usual, $\delta_g$ denotes the image of $g\in C_r$ in $\mathbb{C}[C_r]$.)
Then the Gram matrix $\mathcal{G}$ of $\{g\varphi_i\}_{i\in[n],g\in C_r}$ can be expressed as
\begin{equation}
\label{eq.grampa}
\mathcal{G}
=\sum_{g\in C_r}g^{-1}\lceil \delta_gI\rfloor +\mu\sum_{g\in C_r}g^{-1}\lceil \delta_g\tilde{B}\rfloor.
\end{equation}
Note that this expression leverages our convention that $C_r$ lies in $\mathbb{C}$.

With this, we first show ($\Leftarrow$).
By (a) and (b), we may define $\tilde{B}$ as above, which satisfies (R1)--(R3) by definition.
Then by \eqref{eq.grampa}, the Gram matrix $\mathcal{G}$ of $\{g\varphi_i\}_{i\in[n],g\in C_r}$ carries $\{\lceil \delta_gI\rfloor\}_{g\in C_r}$ and $\{\lceil \delta_g\tilde{B}\rfloor\}_{g\in C_r}$.
By (c), these matrices form an association scheme, and so they span an algebra that is isomorphic to $\mathscr{A}(\tilde{B})$, implying (R4).
As such, $\tilde{B}$ is a roux.
Taking $\alpha \in \hat{C_r}$ to be the identity character, we have that $\langle \varphi_j, \varphi_i \rangle = \mathcal{G}_{(i,1),(j,1)} = \mu\cdot \alpha( \tilde{B}_{ij} )$ for every $i,j \in [n]$.
Hence, $\hat{\alpha}(\tilde{B})$ is the signature matrix of $\{ \varphi_i \}_{i\in [n]}$.
Therefore $\mathscr{L}$ are roux lines for $\tilde{B}$, and its unit norm representatives $\{ \varphi_i \}_{i \in [n]}$ are an equiangular tight frame for their span.
It follows that $\mathcal{G}^2 = c \mathcal{G}$ for some $c>0$, and so Theorem~\ref{thm.roux primitive idempotents} implies that $\mathcal{G}$ is a scalar multiple of a primitive idempotent for the roux scheme of $\tilde{B}$.

For ($\Rightarrow$), there exists an $n\times n$ roux $B$ for some $\Gamma$, and the lines $\mathscr{L}$ have signature matrix $\mathcal{S}=\hat\alpha(B)$ for some $\alpha\in\hat\Gamma$.
That is, there exist unit-norm representatives $\{\varphi_i\}_{i\in[n]}$ of $\mathscr{L}$ whose Gram matrix is $I+\mu\mathcal{S}$ for some $\mu>0$.
Furthermore, the off-diagonal entries of $\mathcal{S}$ lie in the image of $\alpha$, which equals $C_r$ for some $r$.
This gives (a) and (b).
Next, we may define $\tilde{B}$ as above, which by Lemma~\ref{lem.basic roux trans}(c), equals the roux $\bar\alpha(B)$ for $C_r$.
Then \eqref{eq.grampa} shows that the Gram matrix $\mathcal{G}$ of $\{g\varphi_i\}_{i\in[n],g\in C_r}$ carries the corresponding roux scheme, implying (c).
\end{proof}

Signature matrices of unit-norm representatives of roux lines are necessarily comprised of roots of unity, and this feature leads to a necessary integrality condition for the existence of roux lines: 

\begin{corollary}
\label{cor.roux lines integrality}
Suppose there exist $n>d$ roux lines for $\Gamma$ spanning $\mathbb{C}^d$, and put
\[
q=\frac{(n-2d)^2(n-1)}{d(n-d)}.
\]
Then $q\in\mathbb{Z}$ and $\sqrt{q}\in\mathbb{Z}[\omega]$, where $\omega$ is a primitive $r$th root of unity with $r=|\Gamma|$.
\end{corollary}

\begin{proof}
By assumption, there exists an $n\times n$ roux $B$ for $\Gamma$ such that the given lines have unit-norm representatives with signature matrix $\mathcal{S}=\hat\alpha(B)$ for some $\alpha\in\hat\Gamma$.
By Theorem~\ref{thm.roux signature}, $\mathcal{S}$ is the signature matrix of an ETF.
Since ETFs achieve equality in the Welch bound~\eqref{eq.welch bound}, the Gram matrix of this ETF is given by
\[
\mathcal{G}
=I+\sqrt{\frac{n-d}{d(n-1)}}\cdot\mathcal{S},
\]
and tightness implies $\mathcal{G}^2=(n/d)\mathcal{G}$.
We express this quadratic in terms of $\mathcal{S}$ and isolate $\mathcal{S}^2$:
\[
\mathcal{S}^2
=(n-1)I+\operatorname{sign}(n-2d)\cdot\sqrt{q}\cdot\mathcal{S}.
\]
Comparing with \eqref{eq.roux signature matrix}, we note that each $c_g$ is an integer and each $\alpha(g)$ is an $r$th root of unity, and so $\sqrt{q}\in\mathbb{Z}[\omega]$.
This further implies that $\sqrt{q}$ and $q$ are algebraic integers.
Since $q$ is also rational, it must be an integer.
\end{proof}

Recall that Lemma~\ref{lem.basic roux trans} provides a few basic roux transformations.
We now discuss how some of these interact with evaluating a roux at a character.
We say two roux $B,\tilde{B}$ for $\Gamma$ are \textbf{switching equivalent}, denoted $B\sim\tilde{B}$, if there exists a diagonal matrix $D$ as in Lemma~\ref{lem.basic roux trans}(a) such that $\tilde{B}=DBD^{-1}$.
This echoes the more classical notion of switching equivalence between signature matrices, in which the diagonal entries of $D$ are required to be complex with unit modulus.
Note that $B\sim\tilde{B}$ implies that $\hat\alpha(B)$ and $\hat\alpha(\tilde{B})$ are switching equivalent.
(The converse fails to hold by taking $\alpha$ to be defined by $\alpha(z)=1$, for example.)
It is convenient to define the \textbf{normalization} $\bar{B}$ of a roux $B$ for $\Gamma$ to be the unique $\bar{B}\sim B$ with $\bar{B}_{i,1}=\bar{B}_{1,i}=1$ (the identity element of $\Gamma$) for every $i\neq 1$.
Regarding Lemma~\ref{lem.basic roux trans}(d), we note that if $B$ is a roux for $\Gamma\leq\Lambda$, then $\{\hat\alpha(B):\alpha\in\hat\Gamma\}=\{\hat\beta(B):\beta\in\hat\Lambda\}$ since each $\beta\in\hat\Lambda$ restricts to a character $\alpha\in\hat\Gamma$.
In particular, the additional characters in $\hat\Lambda$ fail to produce new roux lines.
The following result reverses the transformation in Lemma~\ref{lem.basic roux trans}(d), and the proof leverages the notion of roux lines:

\begin{lemma}
\label{lem.inverse prob}
Take any $n\times n$ roux $B$ with parameters $\{c_g\}_{g\in\Gamma}$, and put $\Lambda=\langle g:c_g\neq0\rangle$.
Then the normalization of $B$ lies in $\mathbb{C}[\Lambda]^{n\times n}\subseteq\mathbb{C}[\Gamma]^{n\times n}$, and is a roux for $\Lambda$.
Furthermore, if $\tilde{B}\sim B$ is a roux for $\tilde{\Lambda}\leq\Gamma$, then $\Lambda\leq\tilde{\Lambda}$.
\end{lemma}

\begin{proof}
Define $\Pi:=\langle B_{ij}B_{jk}B_{ki}:i\neq j\neq k\neq i\rangle$, and take any $\alpha\in\hat\Gamma$.
Then since $\bar{B}_{i1}=\bar{B}_{1i}=1$ for every $i\neq 1$ and $\bar{B}_{ij}\bar{B}_{jk}\bar{B}_{ki}=B_{ij}B_{jk}B_{ki}$ for every $i,j,k\in[n]$, we have $\bar{B}\in\mathbb{C}[\Pi]^{n\times n}$.
We claim that $\Pi\leq\operatorname{ker}\alpha$ if and only if $\Lambda\leq\operatorname{ker}\alpha$.
This in turn would imply $\Pi=\Lambda$ since a subgroup is determined by its annihilator, and so $\bar{B}\in\mathbb{C}[\Lambda]^{n\times n}$.

We prove our claim by identifying a sequence of equivalent statements.
First, $\Pi\leq\operatorname{ker}\alpha$ if and only if $\hat\alpha(B)_{ij}\hat\alpha(B)_{jk}\hat\alpha(B)_{ki}=1$ whenever $i\neq j\neq k\neq i$.
By Theorem~2.2 in~\cite{ChienW:16}, this is equivalent to $\hat\alpha(B)$ being switching equivalent to $J-I$, where $J$ denotes the $n\times n$ matrix of all ones.
Equivalently, $\hat\alpha(B)$ is the signature matrix of a $1$-dimensional ETF, that is, by Lemma~\ref{lem.two reps of same lines}, we equivalently have $d_{\alpha^{-1}}^+=1$.
By Theorem~\ref{thm.roux primitive idempotents} and Lemma~\ref{lem.B squared}, this is equivalent to having $\hat{c}_\alpha=\hat{c}_{\alpha^{-1}}=n-2$.
Since $\operatorname{Re}\alpha(g)\leq 1$ for every $g\in\Gamma$, Lemma~\ref{lem.B squared} gives
\begin{equation}
\label{eq.calpha bound}
\operatorname{Re}\hat{c}_{\alpha}
=\sum_{g\in\Gamma}c_g\operatorname{Re}\alpha(g)
\leq\sum_{g\in\Gamma}c_g
=n-2,
\end{equation}
with equality only if $\operatorname{Re}\alpha(g)=1$ for every $g$ with $c_g\neq0$, implying $\Lambda\leq\operatorname{ker}\alpha$.
Conversely, $\Lambda\leq\operatorname{ker}\alpha$ implies $\hat{c}_\alpha=\sum_{g\in\Gamma}c_g=n-2$.
Overall, $\hat{c}_{\alpha}=n-2$ if and only if $\Lambda\leq\operatorname{ker}\alpha$, completing the proof of our intermediate claim.

Now that we have $\bar{B}\in\mathbb{C}[\Lambda]^{n\times n}$, we verify that $\bar{B}$ is a roux for $\Lambda$.
To this end, (R1)--(R3) are immediate, while (R4) follows from Lemma~\ref{lem.basic roux trans}(a) and Lemma~\ref{lem.B squared}:
\[
\bar{B}^2
=(n-1)I+\sum_{g\in\Gamma}c_gg\bar{B}
=(n-1)I+\sum_{g\in\Lambda}c_gg\bar{B}.
\]
Overall, $\bar{B}$ is a roux for $\Lambda$.

For the last claim, the previous argument shows that normalizing $\tilde{B}$ produces a roux for $\tilde\Pi=\langle \tilde{B}_{ij}\tilde{B}_{jk}\tilde{B}_{ki}:i\neq j\neq k\neq i\rangle$.
However, $\tilde{B}_{ij}\tilde{B}_{jk}\tilde{B}_{ki}=B_{ij}B_{jk}B_{ki}$ for every $i,j,k\in[n]$, and so $\Lambda=\Pi=\tilde\Pi\leq\tilde\Lambda$, as desired.
\end{proof}

In what follows and throughout, we let $\circ$ denote the \textbf{Hadamard product} defined by $(A\circ B)_{ij}=A_{ij}B_{ij}$, and we let $A^{\circ k}$ denote the $k$th \textbf{Hadamard power} of $A$, defined by $(A^{\circ k})_{ij}=(A_{ij})^k$.

\begin{corollary}[Roux lines detector]
\label{cor.roux lines detector}
Given a signature matrix $\mathcal{S}$, normalize the first row and column to get $\bar{\mathcal{S}}$.
Then $\mathcal{S}$ is the signature matrix of unit-norm representatives of roux lines if and only if the following occur simultaneously:
\begin{itemize}
\item[(a)]
The entries of $\bar{\mathcal{S}}$ are all roots of unity.
\item[(b)]
Every Hadamard power of $\bar{\mathcal{S}}$ has exactly two eigenvalues.
\end{itemize}
\end{corollary}

\begin{proof}
($\Rightarrow$)
Suppose there exists an $n\times n$ roux $B$ for some $\Gamma$, pick $\alpha\in\hat\Gamma$ and let $\mathcal{S}$ be switching equivalent to $\hat\alpha(B)$.
Then $\bar{\mathcal{S}}$ is the normalization of $\hat\alpha(B)$.
Since the off-diagonal entries of $\hat\alpha(B)$ are roots of unity, the same holds for its normalization, implying (a).
Take $D$ such that $\bar{\mathcal{S}}=D\hat\alpha(B)D^{-1}$, and put $v=\operatorname{diag}(D)\in\mathbb{T}^n$.
Then the $k$th Hadamard power of $\bar{\mathcal{S}}$ is given by
\[
\bar{\mathcal{S}}^{\circ k}
=(D\hat\alpha(B)D^{-1})^{\circ k}
=(\hat\alpha(B)\circ vv^*)^{\circ k}
=\widehat{\alpha^k}(B)\circ (v^{\circ k})(v^{\circ k})^*
=D^k\widehat{\alpha^k}(B)(D^k)^{-1}
\]
That is, $\bar{\mathcal{S}}^{\circ k}$ is switching equivalent to $\widehat{\alpha^k}(B)$.
Theorem~\ref{thm.roux signature} then implies (b).

($\Leftarrow$)
Given an $n\times n$ signature matrix $\bar{\mathcal{S}}$ satisfying (a) and (b), pick any $r$ such that the off-diagonal entries of $\bar{\mathcal{S}}$ lie in $C_r$.
Define $B\in\mathbb{C}[C_r]^{n\times n}$ so that $B_{ii}=0$ for every $i\in[n]$ and $B_{ij}=\delta_{\mathcal{S}_{ij}}$ whenever $i\neq j$.
(As usual, $\delta_g$ denotes the image of $g\in C_r$ in $\mathbb{C}[C_r]$.)
We claim that $B$ is a roux, which would imply the result since evaluating $B$ at the character $\alpha$ defined by $\alpha(z)=z$ recovers $\bar{\mathcal{S}}$.
First, $B$ satisfies (R1)--(R3) by definition.
Next, the following holds for every $k$:
\[
\widehat{\alpha^k}(B)
=(\hat\alpha(B))^{\circ k}
=\bar{\mathcal{S}}^{\circ k}.
\]
As such, (b) implies that evaluating $B$ at every character of the form $\alpha^k$ produces the signature matrix of an ETF.
Since $\alpha$ generates $\hat\Gamma$, we may then conclude (R4) by Theorem~\ref{thm.roux signature}.
\end{proof}

We say a sequence of lines is \textbf{real} if their normalized signature matrix is real.
For example, letting $\omega$ denote a primitive cube root of unity, then the lines spanned by $(1,1), (1,\omega), (1,\omega^2)\in\mathbb{C}^2$ are real (even though the Gram matrix of these vectors is not real).

\begin{lemma}[Real lines detector]
\label{lem.real lines detector}
An $n\times n$ signature matrix $\mathcal{S}$ is a signature matrix of real lines if and only if the eigenvalues of $\mathcal{S}^{\circ 2}$ are $n-1$ and $-1$.
\end{lemma}

\begin{proof}
($\Rightarrow$)
Suppose $\mathcal{S}$ is a signature matrix of real lines.
Then the off-diagonal entries of its normalization $\bar{\mathcal{S}}=D^{-1}\mathcal{S}D$ lie in $\{\pm1\}$.
Put $v=\operatorname{diag}(D)$.
Then
\[
\mathcal{S}^{\circ 2}
=(D\bar{\mathcal{S}}D^{-1})^{\circ 2}
=(\bar{\mathcal{S}}\circ vv^*)^{\circ 2}
=(J-I)\circ(v^{\circ 2})(v^{\circ 2})^*
=D^2(J-I)(D^2)^{-1},
\]
i.e., $\mathcal{S}^{\circ 2}$ has the same eigenvalues as $J-I$, where $J$ is the matrix of all ones.

($\Leftarrow$)
Since $\mathcal{S}^{\circ 2}$ has zero trace, the eigenvalues $n-1$ and $-1$ have multiplicities $1$ and $n-1$, respectively.
Thus, $I+\mathcal{S}^{\circ 2}$ has rank $1$ with maximum eigenvalue $n$, and so we may write $I+\mathcal{S}^{\circ 2}=uu^*$ for some $u\in\mathbb{T}^n$.
This in turn implies that $I+\mathcal{S}$ is a solution to $X^{\circ 2}=uu^*$.
Pick any $v\in\mathbb{T}^n$ such that $v^{\circ 2}=u$.
Then every solution has the form $X=vv^*\circ R$, where $R$ has entries in $\{\pm1\}$.
As such, we have $I+\mathcal{S}=vv^*\circ R$ for some symmetric $R\in\{\pm1\}^{n\times n}$ satisfying $R_{ii}=1$ for every $i\in[n]$.
Put $D=\operatorname{diag}(v)$.
Then isolating $\mathcal{S}$ gives
\[
\mathcal{S}
=vv^*\circ R-I
=D(R-I)D^{-1}.
\]
Since $R-I$ is the signature matrix of real lines, we are done.
\end{proof}

\begin{corollary}[Real roux lines detector]
\label{cor.real line detector}
Let $B$ be a roux for $\Gamma$ and pick $\alpha\in\hat\Gamma$.
Then $\hat\alpha(B)$ is a signature matrix of real lines if and only if $\alpha(g)$ is real for every $g\in\Gamma$ such that $c_g\neq0$.
\end{corollary}

\begin{proof}
By Lemma~\ref{lem.real lines detector}, $\hat\alpha(B)$ is a signature matrix of real lines if and only if $\hat\alpha(B)^{\circ 2}$ has minimal polynomial $x^2-(n-2)x-(n-1)$.
By Lemma~\ref{lem.B squared}, evaluating $B$ at any character $\beta\in\hat\Gamma$ gives
\[
(\hat\beta(B))^2
=(n-1)I+\sum_{g\in\Gamma}c_g\beta(g)\hat\beta(B)
=(n-1)I+\hat{c}_{\beta^{-1}}\hat\beta(B)
=(n-1)I+\hat{c}_{\beta}\hat\beta(B)
\]
As such, $\hat\alpha(B)^{\circ 2}=\widehat{\alpha^2}(B)$ has minimal polynomial $x^2-\hat{c}_{\alpha^2}x-(n-1)$; the minimal polynomial does not have degree $1$ since $\hat\alpha(B)^{\circ 2}$ is nonzero with zero trace.
Overall, $\hat\alpha(B)$ is a signature matrix of real lines if and only if $\hat{c}_{\alpha^2}=n-2$.
Finally, the argument in \eqref{eq.calpha bound} gives that $\hat{c}_{\alpha^2}=n-2$ if and only if $\alpha^2(g)=1$ for every $g\in\Gamma$ such that $c_g\neq0$.
\end{proof}

\begin{corollary}
\label{cor.connected and odd implies complex}
Let $B$ be an $n\times n$ roux for $\Gamma$ with parameters $\{c_g\}_{g\in \Gamma}$.
Assume that $\Gamma$ has odd order, and that $\langle g : c_g \neq 0 \rangle = \Gamma$.
Then for every nontrivial $\alpha \in \hat{\Gamma}$, $\hat{\alpha}(B)$ is the signature matrix of non-real lines spanning $\mathbb{C}^d$ for some $d \not\in \{1,n-1\}$.
\end{corollary}

\begin{proof}
To begin, denote $A = \{ g \in \Gamma : c_g \neq 0\}$, and observe that any character $\beta \in \hat{\Gamma}$ with $\beta(g) = 1$ for all $g \in A$ is in fact trivial.
Now let $\alpha \in \hat{\Gamma}$ be nontrivial, and take $\beta = \alpha^2$.
Since $\Gamma$ has odd order, $\beta$ is necessarily nontrivial, and so $\alpha^2(g) \neq 1$ for some $g \in A$.
It follows by Corollary~\ref{cor.real line detector} that $\hat{\alpha}(B)$ is not the signature matrix of real lines.
In particular, any lines with signature matrix $\hat{\alpha}(B)$ span a space of dimension $d \not\in \{1,n-1\}$.
\end{proof}

\section{Roux graphs}
\label{sec.roux graphs}

In this section, we identify graph-theoretic properties that are associated with roux.
We start with a review of certain concepts in graph theory (the reader may consult~\cite{Biggs:74,GodsilH:92,GodsilR:13} for further information).
The graphs in this paper will be assumed to be simple without mention, i.e., they will contain neither loops nor multiple edges.
A graph is said to be \textbf{distance-regular} if for every ordered pair of vertices $(u,v)$, the number of vertices that are simultaneously at distance $i$ from $u$ and distance $j$ from $v$ is determined by $i$, $j$, and the distance between $u$ and $v$.
We say a connected graph is \textbf{antipodal} if the vertices can be partitioned into fibres such that the distance between two distinct vertices is the diameter of the graph if and only if they belong to the same fibre.
A \textbf{cover of the complete graph} consists of a graph and a partition of its vertices into fibres of independent sets such that the induced subgraph on the union of any two distinct fibres is a perfect matching.
(In particular, any two fibres have the same size.)
Finally, an \textbf{antipodal cover of the complete graph} is a cover that is an antipodal graph such that the covering fibres coincide with the antipodal fibres.
In what follows, we consider distance-regular antipodal covers of the complete graph (\textsc{drackn}s).

\begin{proposition}[Lemma~3.1 in~\cite{GodsilH:92}]
For every distance-regular antipodal cover $\mathscr{G}$ of the complete graph, there exist constants $(n,r,c)$ for which $\mathscr{G}$ is a connected graph on $rn$ vertices such that
\begin{itemize}
\item[(D1)]
every pair of vertices at distance $2$ has $c$ common neighbors,
\item[(D2)]
the vertices can be partitioned into fibres of size $r$ such that the distance between two vertices is the diameter of the graph if and only if they belong to the same fibre, and
\item[(D3)]
the induced subgraph between any two distinct fibres is a perfect matching.
\end{itemize}
Conversely, any connected graph on $rn$ vertices satisfying (D1)--(D3) for some $(n,r,c)$ is a distance-regular antipodal cover of the complete graph.
\end{proposition}

We refer to the corresponding constants $(n,r,c)$ above as the \textit{parameters} of the \textsc{drackn}.

Given a cover of the complete graph, the fibres suggest a block-matrix expression for the $rn\times rn$ adjacency matrix.
In particular, the $r\times r$ blocks are zero on the diagonal and permutation matrices off the diagonal.
It is convenient to apply a particular graph isomorphism that we refer to as \textbf{normalization}.
Explicitly, let $A_{1j}\in\mathbb{R}^{r\times r}$ denote the $j$th block in the first block row of the adjacency matrix $A$.
Consider the block diagonal matrix $D$ whose $j$th diagonal block is the $r\times r$ identity for $j=1$, and $A_{1j}$ for $j>1$.
We conjugate to obtain the normalized adjacency matrix $DAD^{-1}$, whose off-diagonal blocks in the first block row (and column) are all identity matrices.
Notice that normalization is not uniquely determined up to isomorphism of the cover of the complete graph, but the permutation group $\Gamma\leq S_r$ generated by the permutations represented by the off-diagonal blocks of the normalized adjacency matrix is uniquely determined up to conjugation within $S_r$.
If $\Gamma$ acts regularly on $[r]$, then we say the cover is \textbf{regular}.
If $\Gamma$ is abelian, then we say the cover is \textbf{abelian}.
(We note that our usage of \textit{abelian} is different from the terminology in~\cite{GodsilH:92,CoutinhoCSZ:16}, which calls a cover \textit{abelian} if it is both regular and abelian; we use \textit{regular abelian} in such a setting to highlight both properties of $\Gamma$.
Moreover, in~\cite{GodsilH:92} a \emph{regular cover} is assumed to be a distance-regular graph, but for us it is not.)

In this section, we are particularly interested in regular abelian \textsc{drackn}s and their relationship with roux.
In~\cite{CoutinhoCSZ:16}, Coutinho, Godsil, Shirazi and Zhan establish that, after normalizing the adjacency matrix of a regular abelian \textsc{drackn} and then replacing blocks with corresponding members of $\mathbb{C}[\Gamma]$ to produce a matrix $B\in\mathbb{C}[\Gamma]^{n\times n}$, evaluating the entries of $B$ at any character of $\Gamma$ produces the $n\times n$ signature matrix of an ETF (see also~\cite{FickusJMPW:17}).
This behavior of regular abelian \textsc{drackn}s should be compared with Theorem~\ref{thm.roux signature}.
In a roux scheme, the adjacency matrix $\lceil B\rfloor$ is symmetric by (R3), and therefore describes a graph we call a \textbf{roux graph}.
Notice that the resulting vertex set is $[n] \times \Gamma$, and the roux graph is a cover of the complete graph $K_n$ with fibres of size $|\Gamma|$, each having the form $\{i\} \times \Gamma$.
The following result identifies the relationship between roux graphs and regular abelian \textsc{drackn}s.

\begin{theorem}\
\label{thm.drackn vs roux}
\begin{itemize}
\item[(a)]
Every regular abelian $(n,r,c)$-\textsc{drackn} is a roux graph for a roux with parameters
\begin{equation}
\label{eq.drackn roux parameters}
c_g=\left\{\begin{array}{ll}
n-c(r-1)-2&\text{if }g=\operatorname{id};\\
c&\text{otherwise.}
\end{array}\right.
\end{equation}
\item[(b)]
Given a finite abelian group $\Gamma$, then $B\in\mathbb{C}[\Gamma]^{n\times n}$ satisfies (R1)--(R3) if and only if $\lceil B\rfloor$ is the adjacency matrix of a (not necessarily regular) abelian cover of the complete graph.
\item[(c)]
Given a roux with parameters $\{c_g\}_{g\in\Gamma}$, its roux graph has diameter $3$ if and only if $c_g\neq 0$ for every $g\in\Gamma$, in which case the graph is a regular antipodal cover of the complete graph.
\item[(d)]
The roux graph of any roux with parameters \eqref{eq.drackn roux parameters} and $c \neq 0$ is a regular abelian $(n,r,c)$-\textsc{drackn}.
\end{itemize}
\end{theorem}

\begin{proof}
For (a), we may normalize the adjacency matrix to obtain $A$, whose off-diagonal blocks are represented by permutations that generate a group $\Gamma\leq S_r$.
Replacing blocks of $A$ with corresponding members of $\mathbb{C}[\Gamma]$ produces the matrix $B\in\mathbb{C}[\Gamma]^{n\times n}$.
We will use a result from~\cite{GodsilH:92}, but first, we translate notation: our $B_{ij}$ equals their $f(i,j)$, and our $B$ equals their $A^f$.
With this in mind, Corollary~7.5 in~\cite{GodsilH:92} gives
\[
B^2
=(n-1)I+(a_1-c)B+c\sum_{g\in\Gamma}g(J-I)
=(n-1)I+\sum_{g\in\Gamma}c_ggB,
\]
with $c_\mathrm{id}=a_1=n-c(r-1)-2$ and $c_g=c$ otherwise.
By Lemma~\ref{lem.B squared}, $B$ is a roux.
Next, (b) is immediate.
For (c), we first apply Lemma~\ref{lem.B squared} to get
\begin{align}
\label{eq.roux graph squared}
\lceil B\rfloor^2
&=(n-1)\lceil I\rfloor+\sum_{g\in\Gamma}c_g\lceil gB\rfloor,\\
\label{eq.roux graph cubed}
\lceil B\rfloor^3
&=(n-1)\lceil B\rfloor+\sum_{g\in\Gamma}c_g\lceil gI\rfloor+\sum_{g,h\in\Gamma}c_gc_h\lceil ghB\rfloor.
\end{align}
Then (R1) and \eqref{eq.roux graph squared} give that distinct vertices in a common fibre have distance at least $3$.
Furthermore, \eqref{eq.roux graph cubed} implies that all such vertices have distance $3$ precisely when every $c_g$ is strictly positive.
This proves ($\Rightarrow$).
For ($\Leftarrow$), it remains to show that points in different fibres have distance at most $3$ when every $c_g$ is strictly positive.
In fact, \eqref{eq.roux graph squared} gives that all such vertices have distance at most $2$.
This stronger conclusion implies that the roux graph is an antipodal cover.
To see that we have a regular cover, consider the normalization $\overline{B}=DBD^*$ of $B$.
Here, $\lceil \overline{B}\rfloor$ is the normalization of $\lceil B\rfloor$ since $\lceil \cdot\rfloor$ is a $*$-homomorphism.
Lemma~\ref{lem.inverse prob} implies that the off-diagonal entries of $\overline{B}$ generate all of $\Gamma$ since $c_g \neq 0$ for all $g\in\Gamma$.
Also, $\Gamma$ acts regularly on itself, and so $\lceil \overline{B}\rfloor$ is the adjacency matrix of a regular cover, as desired.
Finally, for (d), notice that the roux graph is a regular abelian antipodal cover of the complete graph by (c).
It remains to prove it is distance regular, and this follows from Lemma~3.1 in~\cite{GodsilH:92}.
\end{proof}

At this point, we identify an example that demonstrates that the theory of roux extends beyond \textsc{drackn}s and Higman pairs.
In particular, note that the roux constructed from antisymmetric conference matrices in Lemma~\ref{lem.conference roux} do not arise from \textsc{drackn}s since the $c_g$'s for non-identity $g\in C_4$ are not all equal.
Consider the following iterative construction of antisymmetric conference matrices (based on Theorem~14 in~\cite{KoukouvinosS:08}):
\[
M_1:=\left[\begin{array}{rr}0&1\\-1&0\end{array}\right],
\qquad
M_{k+1}:=\left[\begin{array}{cc}M_k&M_k+I\\M_k-I&-M_k\end{array}\right].
\]
One may verify in GAP~\cite{GAP:software,Hanaki:online} that the roux corresponding to $M_4$ is not Schurian.
By Theorem~\ref{thm.schurian roux}, the roux scheme does not arise from a Higman pair.

\begin{theorem}[Regular abelian \textsc{drackn}s from roux]
\label{thm.drackns from roux}
Given a connected $\Gamma$-roux graph of order $n|\Gamma|$, then for every odd prime $p$ dividing $|\Gamma|$, there exists a regular abelian $(n,p,c)$-\textsc{drackn} for some $c$.
\end{theorem}

The proof of this theorem leverages the spectrum of roux graphs.

\begin{theorem}
\label{thm.roux eigenvalues}
Given an $n\times n$ roux over an abelian group $\Gamma$ and having parameters $c=\{c_g\}_{g\in\Gamma}$, the corresponding roux graph has eigenvalues $\lambda_{\alpha}^\epsilon$
given by
\[
\lambda_{\alpha}^\epsilon
=\frac{\hat{c}_\alpha+\epsilon\sqrt{(\hat{c}_\alpha)^2+4(n-1)}}{2},
\qquad
(\alpha\in\hat\Gamma,~\epsilon\in\{+,-\}).
\]
Furthermore, the mapping $(\hat{c}_\alpha,\epsilon) \mapsto \lambda_\alpha^\epsilon$ is injective on $\{ \hat{c}_\alpha : \alpha \in \hat{\Gamma} \} \times \{ +,- \}$.
\end{theorem}

\begin{proof}
Let $B$ denote the underlying roux.
Then projection onto any eigenspace of $\lceil B\rfloor$ is an idempotent of the corresponding roux scheme.
Considering Theorem~\ref{thm.roux primitive idempotents}, the eigenvalues of $\lceil B\rfloor$ are therefore the scalars $\lambda_\alpha^\epsilon$ such that $\lceil B\rfloor\mathcal{G}_\alpha^\epsilon=\lambda_\alpha^\epsilon\mathcal{G}_\alpha^\epsilon$.
The definition of $\mathcal{G}_\alpha^\epsilon$ and Lemma~\ref{lem.B squared} together give $\operatorname{tr}(\lceil B\rfloor\mathcal{G}_\alpha^\epsilon)=rn(n-1)\mu_\alpha^\epsilon=(n-1)\mu_\alpha^\epsilon\operatorname{tr}(\mathcal{G}_\alpha^\epsilon)$, and so we must have $\lambda_\alpha^\epsilon=(n-1)\mu_\alpha^\epsilon$.
The definition of $\mu_\alpha^\epsilon$ in Theorem~\ref{thm.roux primitive idempotents} then gives the formula for the eigenvalues.
For injectivity, choose $\epsilon \in \{ +, - \}$ and observe that the function $f \colon \mathbb{R} \to \mathbb{R}$ given by $f(x):=\frac{1}{2}( x + \epsilon \sqrt{x^2 + 4(n-1)} )$ is one-to-one since $2f'(x) = 1 + \epsilon x (x^2+4(n-1) )^{-1/2} > 0$.
In particular, $\lambda_\alpha^\epsilon \neq \lambda_\beta^\epsilon$ whenever $\hat{c}_\alpha \neq \hat{c}_\beta$.
Injectivity of the mapping $(\hat{c}_\alpha,\epsilon) \mapsto \lambda_\alpha^\epsilon$ follows since $\lambda_\alpha^+ > 0 \geq \lambda_\beta^-$ for any $\alpha,\beta \in \hat{\Gamma}$.
\end{proof}

In words, the above theorem gives that the spectrum of a roux graph is a nonlinear function of the Fourier spectrum of the parameters of the underlying roux.
As is standard in spectral graph theory, the spectrum can be leveraged to identify combinatorial structure in the graph.
See Theorem~\ref{thm.four eigs} below for an instance of this phenomenon.

\begin{corollary}
\label{cor.connectedness}
Let $B$ be a roux for $\Gamma$ with parameters $\{c_g\}_{g\in \Gamma}$, and let $\Lambda = \langle g \colon c_g \neq 0 \rangle$.
Then the corresponding roux graph has exactly $[ \Gamma : \Lambda ]$ connected components.
Furthermore, the following are equivalent:
\begin{itemize}
\item[(a)]
The roux graph of $B$ is connected.
\item[(b)]
The roux graph of $B$ is a regular abelian cover of the complete graph.
\item[(c)]
$\Lambda = \Gamma$.
\item[(d)]
$B$ is not switching equivalent to a roux for any proper subgroup of $\Gamma$.
\end{itemize}
\end{corollary}

\begin{proof}
Denote $n$ for the size of $B$ and $\mathscr{G}$ for the corresponding roux graph.
We begin by verifying the formula for the number of connected components.
Since $\mathscr{G}$ is an $(n-1)$-regular graph, it suffices to compute the multiplicity of the eigenvalue $n-1$. 
Using the notation of Theorem~\ref{thm.roux eigenvalues}, it is straightforward to verify that $\lambda_\alpha^\epsilon = n-1$ if and only if $\epsilon = +$ and $\hat{c}_\alpha = n-2$.
As in \eqref{eq.calpha bound}, the latter happens if and only if $\alpha(g) = 1$ whenever $c_g \neq 0$, if and only if $\alpha$ lies in the annihilator $\Lambda^* \leq\hat{\Gamma}$.
Consequently, the multiplicity of $n-1$ as an eigenvalue of $\mathscr{G}$ equals $|\Lambda^*| = [\Gamma : \Lambda]$.
This gives the desired formula for the number of connected components.
The equivalence of (a), (c), and (d) follows immediately from Lemma~\ref{lem.inverse prob}.
Furthermore, (c) implies (b) as in the proof of Theorem~\ref{thm.drackn vs roux}(c).
Conversely, suppose (b) holds, and let $\overline{B}$ be the normalization of $B$.
By Lemma~\ref{lem.inverse prob}, the off-diagonal entries of $\overline{B}$ generate $\Lambda$.
Since $\lceil \overline{B} \rfloor$ is the normalization of $\lceil B \rfloor$, it follows that $\Lambda$ has a regular permutation representation on $r$ points.
In particular, $|\Lambda| = r = |\Gamma|$.
This proves (c).
\end{proof}

\begin{theorem}
\label{thm.four eigs}
Let $\mathscr{G}$ be a roux graph that is connected, antipodal, and not the complete graph.
Then $\mathscr{G}$ is distance regular if and only if it has exactly four eigenvalues.
\end{theorem}

\begin{proof}
($\Rightarrow$)
Corollary~\ref{cor.connectedness} implies that $\mathscr{G}$ is a regular abelian cover of the complete graph.
By Theorem~\ref{thm.drackn vs roux}(a), $\mathscr{G}$ could be obtained from a roux with parameters \eqref{eq.drackn roux parameters}.
Taking the Fourier transform gives that there exists $t$ such that
\begin{equation}
\label{eq.fft of drackn parameters}
\hat{c}_\alpha
=\left\{\begin{array}{cl}
n-2&\text{if }\alpha=1;\\
t&\text{otherwise}.
\end{array}\right.
\end{equation}
We claim that $t \neq n-2$, in which case $\hat{c}_\alpha$ has exactly two values, and so Theorem~\ref{thm.roux eigenvalues} gives that $\mathscr{G}$ has exactly four eigenvalues.
Otherwise, taking the inverse Fourier transform shows that $c_g = 0$ for every $g \neq \operatorname{id}$.
In that case, the adjacency matrix $A$ of $\mathscr{G}$ satisfies $A^2 \in \operatorname{span}\{I,A\}$, and consequently $A^k\in \operatorname{span}\{I,A\}$ for every $k > 1$.
This is impossible since $\mathscr{G}$ is connected and not complete.

($\Leftarrow$)
Let $B \in \mathbb{C}[\Gamma]^{n\times n}$ be a roux with parameters $\{c_g\}_{g\in \Gamma}$ whose roux graph $\mathscr{G}$ is connected, antipodal, not complete, and has exactly four eigenvalues.
The graph is $(n-1)$-regular, and so $n-1$ is an eigenvalue with all-ones eigenvector.
In fact, $\lambda_1^+=n-1$ since $\hat{c}_1=n-2$.
The multiplicity of $n-1$ is $1$ since $\mathscr{G}$ is connected, and all other eigenvalues $\lambda$ satisfy $|\lambda|<n-1$.
Since $\hat{c}_\alpha\mapsto\lambda_\alpha^+$ in Theorem~\ref{thm.roux eigenvalues} is strictly increasing, we therefore have $\hat{c}_\alpha<n-2$ for every $\alpha\neq 1$.
Furthermore, by Theorem~\ref{thm.roux eigenvalues}, every value of $\hat{c}_\alpha$ produces two distinct eigenvalues.
Since our roux graph has exactly four eigenvalues, there must be exactly two distinct values of $\hat{c}_\alpha$, that is, there exists $t$ such that \eqref{eq.fft of drackn parameters} holds.
As above, $t \neq n-2$ since $\mathscr{G}$ is connected and not complete.
Applying the inverse Fourier transform then produces the roux parameters of a regular abelian \textsc{drackn} with $c \neq 0$, and so we are done by Theorem~\ref{thm.drackn vs roux}(d).
\end{proof}

\begin{proof}[Proof of Theorem~\ref{thm.drackns from roux}]
Let $B$ denote the underlying roux for $\Gamma$.
As a consequence of the classification of finitely generated abelian groups, there exists a surjection $\varphi\colon\Gamma\to C_p$.
By Lemma~\ref{lem.basic roux trans}(c), we have that $\bar\varphi(B)$ is a roux for $C_p$ with parameters $\bar{c}_\lambda=\sum_{g\in\varphi^{-1}(\lambda)}c_g$ for $\lambda\in C_p$.
Furthermore, we must have $\bar{c}_\lambda\neq0$ for some $\lambda\neq0$ since otherwise $\langle c_g:g\neq0\rangle\leq\operatorname{ker}\varphi\lneq\Gamma$, which contradicts connectedness by Corollary~\ref{cor.connectedness}.
Since $p$ is odd, Corollary~\ref{cor.connected and odd implies complex} implies that any nontrivial $\beta\in\widehat{C_p}$ produces a signature matrix $\mathcal{S}=\hat\beta(\bar\varphi(B))$ of an ETF in $\mathbb{C}^d$ for some $d\neq1$.
In particular, the corresponding equiangular lines saturate the relative bound.
Finally, the off-diagonal entries of $\mathcal{S}$ are $p$th roots of unity, and so Theorem~5.1 in~\cite{CoutinhoCSZ:16} gives the result.
\end{proof}

\section{Applications and examples}
\label{sec.apps and exs}

In this section, we identify consequences of our theory for regular abelian covers of the complete graph and maximal equiangular tight frames.
Additionally, we summarize every construction of doubly transitive lines that we currently know.

\subsection{Consequences for regular abelian covers of the complete graph}

We start by observing that regular abelian \textsc{drackn}s satisfy a stronger version of Theorem~\ref{thm.roux signature}:

\begin{theorem}
\label{thm.drackn signature}
Suppose $B\in\mathbb{C}[\Gamma]^{n\times n}$ satisfies (R1)--(R3).
Then $\lceil B\rfloor$ is the adjacency matrix of a regular abelian $(n,r,c)$-\textsc{drackn} if and only if for every $\alpha\in\hat{\Gamma}$, $\hat\alpha(B)$ is the signature matrix of an equiangular tight frame for $\mathbb{C}^d$, where
\[
d:=\frac{n}{1+(n-1)\mu^2},
\qquad
\mu:=\frac{\delta+\sqrt{\delta^2+4(n-1)}}{2(n-1)},
\qquad
\delta:=n-rc-2.
\]
\end{theorem}

Note that ($\Rightarrow$) corresponds to Theorem~4.1 in~\cite{CoutinhoCSZ:16}, but our theory allows for a quick proof of both directions simultaneously.

\begin{proof}[Proof of Theorem~\ref{thm.drackn signature}]
By Theorem~\ref{thm.drackn vs roux}, $\lceil B\rfloor$ is the adjacency matrix of a regular abelian \textsc{drackn} if and only if $B$ has roux parameters
\[
c_g
=\left\{\begin{array}{ll}
n-c(r-1)-2&\text{if }g=\operatorname{id};\\
c&\text{otherwise,}
\end{array}\right.
\]
if and only if $\hat{c}$ is given by
\begin{equation}
\label{eq.fourier of drackn parameters}
\hat{c}_\alpha
=\left\{\begin{array}{ll}
n-2&\text{if }\alpha=1;\\
\delta&\text{otherwise.}
\end{array}\right.
\end{equation}
Recall that $\hat{c}_\alpha\mapsto\lambda_\alpha^+$ is injective by Theorem~\ref{thm.roux eigenvalues}.
Also, we have $\lambda_\alpha^+= (n-1)\mu_\alpha^+>0$, and furthermore, the mapping $\mu_\alpha^+\mapsto d_\alpha^+$ given in Theorem~\ref{thm.roux primitive idempotents} is strictly increasing.
It follows that $\hat{c}_\alpha\mapsto d_\alpha^+$ is also injective.
Thus, \eqref{eq.fourier of drackn parameters} holds if and only if $B$ is a roux with $d_\alpha^+=d$ for all nontrivial $\alpha\in\hat\Gamma$ (by Theorem~\ref{thm.roux primitive idempotents}), if and only if for every $\alpha\in\hat{\Gamma}$, $\hat\alpha(B)$ is the signature matrix of an equiangular tight frame for $\mathbb{C}^d$ (by Theorem~\ref{thm.roux signature} and Lemma~\ref{lem.two reps of same lines}).
\end{proof}

Much like roux lines, we say lines are \textsc{drackn} if their signature matrix can be obtained by evaluating a regular abelian \textsc{drackn}'s roux at a character (as suggested by the previous theorem).

\begin{corollary}[\textsc{drackn} lines detector]
\label{cor.drackn lines detector}
Let $\mathscr{L}$ be a sequence of lines having normalized signature matrix $\mathcal{S}$.
Then $\mathscr{L}$ forms \textsc{drackn} lines if and only if
\begin{itemize}
\item[(a)]
every off-diagonal entry of $\mathcal{S}$ is an $r$th root of unity for some minimal $r>0$, and
\item[(b)]
there exist $\theta$ and $\tau$ such that for every $k\in\{1,\ldots,r-1\}$, $\mathcal{S}^{\circ k}$ has spectrum $\{\theta,\tau\}$.
\end{itemize}
\end{corollary}

\begin{proof}
($\Rightarrow$)
Let $B$ denote the underlying roux, and let $\Gamma$ be the underlying abelian group.
Notice that $\mathcal{S}=\hat\alpha(B)$ for some character $\alpha\in\hat\Gamma$, and $\mathcal{S}^{\circ k}=\widehat{\alpha^k}(B)$ for every $k$.
By Corollary~\ref{cor.roux lines detector}, it suffices to show that $\hat\beta(B)$ has the same minimal polynomial for every nontrivial character $\beta\in\hat\Gamma$.
Since $\hat{c}_\beta$ has the form \eqref{eq.fft of drackn parameters}, then Lemma~\ref{lem.B squared} gives $(\hat\beta(B))^2=(n-1)I+t\hat\beta(B)$ for every nontrivial $\beta$.
Since each $\hat\beta(B)$ is nonzero with zero trace, we then have that the minimal polynomial of $\hat\beta(B)$ is $x^2-tx-(n-1)$ for every nontrivial $\alpha$, as desired.

($\Leftarrow$)
By Corollary~\ref{cor.roux lines detector}, $\mathscr{L}$ forms roux lines for the group $C_r$ with roux $B\in\mathbb{C}[C_r]^{n\times n}$ defined by $B_{ii}=0$ for $i\in[n]$ and $B_{ij}=\delta_{\mathcal{S}_{ij}}$ for $i,j\in[n]$ with $i\neq j$.
(As usual, $\delta_g$ denotes the image of $g\in C_r$ in $\mathbb{C}[C_r]$.)
The proof of Theorem~\ref{thm.roux signature} (specifically \eqref{eq.etf to roux 2}) gives that $(\hat\alpha(B))^2=(n-1)I+\hat{c}_\alpha\hat\alpha(B)$ for every character $\alpha\in\hat\Gamma$.
Since each $\hat\alpha(B)$ is nonzero with zero trace, we then have that the minimal polynomial of $\hat\alpha(B)$ is $x^2-\hat{c}_\alpha x-(n-1)$ for every $\alpha\in\hat\Gamma$.
By assumption, this minimal polynomial is the same for every nontrivial $\alpha$, and so $\hat{c}_\alpha$ has the form \eqref{eq.fft of drackn parameters}.
Applying the inverse Fourier transform then produces the roux parameters of~\eqref{eq.drackn roux parameters}, and so $B$ defines a regular abelian \textsc{drackn} by Theorem~\ref{thm.drackn vs roux}(d).
\end{proof}

Any signature matrix $\mathcal{S}$ comprised of prime roots of unity that has a quadratic minimal polynomial $p\in\mathbb{Q}[x]$ necessarily satisfies (b) above, and therefore corresponds to a regular abelian \textsc{drackn}.
Indeed, in this case, taking the $k$th Hadamard power of $\mathcal{S}$ is equivalent to applying a field automorphism of $\mathbb{Q}(e^{2\pi\mathrm{i}/r})$ entrywise, which fixes $p$.
While signature matrices from \textsc{drackn} lines necessarily have a quadratic minimal polynomial in $\mathbb{Q}[x]$, namely $x^2-(n-rc-2)x-(n-1)$, one may remove the polynomial's rationality from the hypothesis here (see Theorem~5.1 in~\cite{CoutinhoCSZ:16}).
Comparing Corollary~\ref{cor.drackn lines detector} with Corollary~\ref{cor.roux lines detector}, we see that \textsc{drackn} lines are the roux lines for which Hadamard powers of the normalized signature matrix correspond to ETFs in a common dimension.
In this sense, this completes the picture of ``lines from covers'' and ``covers from lines'' introduced in~\cite{CoutinhoCSZ:16}.
Next, we provide a stronger version of Corollary~\ref{cor.roux lines integrality} for \textsc{drackn} lines:

\begin{corollary}
\label{cor.drackn lines integrality}
Suppose there exist $n>d$ \textsc{drackn} lines for $\Gamma$ spanning $\mathbb{C}^d$, and put
\[
q=\frac{(n-2d)^2(n-1)}{d(n-d)}.
\]
Then $\sqrt{q}\in\mathbb{Z}$.
\end{corollary}

\begin{proof}
Let $B$ be a roux for $\Gamma$ whose roux graph is a regular abelian \textsc{drackn}, and let $d\neq 1$ be the constant dimension for which $\hat\alpha(B)$ is the signature matrix of an ETF in $\mathbb{C}^d$ for every nontrivial $\alpha\in\hat\Gamma$ (given in Theorem~\ref{thm.drackn signature}).
Following the proof of Corollary~\ref{cor.roux lines integrality}, we have
\[
(\hat\alpha(B))^2
=(n-1)I+\operatorname{sign}(n-2d)\cdot\sqrt{q}\cdot\hat\alpha(B).
\]
Comparing with \eqref{eq.roux signature matrix} then gives
\[
\hat{c}_\alpha
=\left\{\begin{array}{ll}
n=2&\text{if }\alpha=1;\\
\operatorname{sign}(n-2d)\sqrt{q}&\text{otherwise.}
\end{array}\right.
\]
We invert the Fourier transform to get
\[
c_g
=\frac{1}{|\Gamma|}\sum_{\alpha\in\hat\Gamma}\hat{c}_\alpha \alpha(g)
=\frac{1}{|\Gamma|}\Big(n-2-\operatorname{sign}(n-2d)\sqrt{q}\Big),
\qquad
(g\neq 1).
\]
Since $c_g$ is an integer by Lemma~\ref{lem.B squared}, we are done.
\end{proof}

Corollary~\ref{cor.drackn lines integrality} rules out the existence of many regular abelian \textsc{drackn}s.
For example, Table~\ref{table.drackns} lists regular abelian \textsc{drackn} parameters that meet the necessary conditions with $n\leq 500$ and $r$ an odd prime.
For the sake of reproducibility, we provide our methodology for constructing this table:
\begin{itemize}
\item[1.]
Find pairs $(d,n)$ with $n \leq 500$ satisfying Corollary~\ref{cor.drackn lines integrality}.
By Corollary~\ref{cor.connected and odd implies complex}, we may ignore $d \in \{1,n-1\}$.
Following Gerzon's bound~\cite{LemmensS:73}, we also require $n \leq \min\{d^2,(n-d)^2\}$.
\item[2.]
Find odd primes $r$ dividing $n$ for which there exists $c \in \mathbb{Z}$ corresponding to $d$.
The fact that $r$ necessarily divides $n$ is given by Theorem~9.2 in~\cite{GodsilH:92}.
\item[3.]
Check additional constraints from~\cite{GodsilH:92}, as summarized by Theorem~3.1 in~\cite{CoutinhoCSZ:16}.
\end{itemize}
We note that Theorem~\ref{thm.drackns from roux} establishes how Table~\ref{table.drackns} can be used to preclude the existence of regular abelian \textsc{drackns} (and more generally, connected roux graphs) over groups of odd order.
For example, since $n=64$ does not appear in Table~\ref{table.drackns}, any connected roux graph with $n=64$ must necessarily be over a group $\Gamma$ whose order is a power of $2$.
(In fact, the next subsection constructs such a roux with $\Gamma=C_4$.)
As another perspective, Table~\ref{table.drackns} and Theorem~\ref{thm.drackns from roux} together indicate several directions for future research.
Indeed, if there is a connected roux graph with $n \leq 500$ for an abelian group $\Gamma$ of order other than a power of $2$, then $n$ must appear in a row of Table~\ref{table.drackns} with every odd prime $r$ dividing $|\Gamma|$.
As such, repeated values of $n$ suggest the possible existence of roux for groups of composite order.

\begin{table}
\begin{center}
\begin{tabular}{rrrrrl}
$d$ & $n$ & $r$ & $c$ & $\delta$ & Existence \\ \hline
6 & 9 & 3 & 3 & $-2$ & \cite{FickusJMPW:17, Godsil:93, KlinPech:11, Tsiovkina:17} \\
15 & 25 & 5 & 5 & $-2$ & \cite{FickusJMPW:17, Godsil:93, KlinPech:11} \\
11 & 33 & 3 & 9 & 4 \\
21 & 36 & 3 & 12 & $-2$ & \cite{KlinPech:11} \\
12 & 45 & 3 & 12 & 7 & \cite{KlinPech:11}\\
33 & 45 & 5 & 10 & $-7$ \\
28 & 49 & 7 & 7 & $-2$ & \cite{FickusJMPW:17, Godsil:93, KlinPech:11} \\
34 & 51 & 3 & 18 & $-5$ \\
22 & 55 & 5 & 10 & 3 & \cite{FickusJMP:18} \\
52 & 65 & 5 & 15 & $-12$ & \cite{FickusJMPW:17, Tsiovkina:17}\\
45 & 81 & 3 & 27 & $-2$ & \cite{FickusJMPW:17, KlinPech:11} \\
65 & 91 & 7 & 14 & $-9$ \\
76 & 96 & 3 & 36 & $-14$ \\
33 & 99 & 3 & 30 & 7 \\
55 & 100 & 5 & 20 & $-2$ & \cite{KlinPech:11}\\
14 & 105 & 3 & 27 & 22 \\
40 & 105 & 7 & 14 & 5 & \cite{FickusJMP:18} \\
65 & 105 & 3 & 36 & $-5$ \\
35 & 120 & 3 & 36 & 10 \\
66 & 121 & 11 & 11 & $-2$ & \cite{FickusJMPW:17, Godsil:93, KlinPech:11} \\
105 & 126 & 3 & 48 & $-20$ & \cite{FickusJMPW:17, Tsiovkina:17} \\
86 & 129 & 3 & 45 & $-8$ \\
78 & 144 & 3 & 48 & $-2$ & \cite{KlinPech:11} \\
29 & 145 & 5 & 25 & 18 \\
46 & 161 & 7 & 21 & 12 \\
91 & 169 & 13 & 13 & $-2$ & \cite{FickusJMPW:17, Godsil:93, KlinPech:11} \\
30 & 175 & 5 & 30 & 23 \\
145 & 175 & 7 & 28 & $-23$ \\
133 & 190 & 5 & 40 & $-12$ \\
105 & 196 & 7 & 28 & $-2$ & \cite{KlinPech:11} \\
67 & 201 & 3 & 63 & 10 \\
77 & 210 & 5 & 40 & 8 \\
133 & 210 & 3 & 72 & $-8$ \\
186 & 217 & 7 & 35 & $-30$ \\
120 & 225 & 3 & 75 & $-2$ \\
120 & 225 & 5 & 45 & $-2$ \\
175 & 225 & 3 & 81 & $-20$ \\
70 & 231 & 3 & 72 & 13 \\
161 & 231 & 11 & 22 & $-13$ \\
\end{tabular}
\qquad
\begin{tabular}{rrrrrl}
$d$ & $n$ & $r$ & $c$ & $\delta$ & Existence \\ \hline
162 & 243 & 3 & 84 & $-11$ \\
41 & 246 & 3 & 72 & 28 \\
92 & 253 & 11 & 22 & 9 & \cite{FickusJMP:18} \\
52 & 273 & 7 & 35 & 26 \\
221 & 273 & 3 & 99 & $-26$ \\
217 & 280 & 5 & 60 & $-22$ \\
42 & 288 & 3 & 84 & 34 \\
153 & 289 & 17 & 17 & $-2$ & \cite{FickusJMPW:17, Godsil:93, KlinPech:11} \\
177 & 295 & 5 & 60 & $-7$ \\
129 & 301 & 7 & 42 & 5 \\
88 & 320 & 5 & 60 & 18 \\
171 & 324 & 3 & 108 & $-2$ & \cite{KlinPech:11} \\
225 & 325 & 13 & 26 & $-15$ \\
260 & 325 & 5 & 70 & $-27$ \\
113 & 339 & 3 & 108 & 13 \\
78 & 351 & 3 & 108 & 25 \\
126 & 351 & 13 & 26 & 11 & \cite{FickusJMP:18} \\
225 & 351 & 3 & 120 & $-11$ \\
190 & 361 & 19 & 19 & $-2$ & \cite{FickusJMPW:17, Godsil:93, KlinPech:11} \\
117 & 378 & 3 & 120 & 16 \\
261 & 378 & 7 & 56 & $-16$ \\
33 & 385 & 5 & 65 & 58 \\
55 & 385 & 7 & 49 & 40 \\
105 & 385 & 11 & 33 & 20 \\
154 & 385 & 5 & 75 & 8 \\
262 & 393 & 3 & 135 & $-14$ \\
210 & 400 & 5 & 80 & $-2$ & \cite{KlinPech:11} \\
145 & 406 & 7 & 56 & 12 \\
56 & 441 & 7 & 56 & 47 \\
231 & 441 & 3 & 147 & $-2$ \\
231 & 441 & 7 & 63 & $-2$ \\
385 & 441 & 3 & 162 & $-47$ \\
369 & 451 & 11 & 44 & $-35$ \\
391 & 460 & 5 & 100 & $-42$ \\
370 & 481 & 13 & 39 & $-28$ \\
253 & 484 & 11 & 44 & $-2$ & \cite{KlinPech:11} \\
97 & 485 & 5 & 90 & 33 \\
209 & 495 & 3 & 162 & 7 \\
286 & 495 & 5 & 100 & $-7$ \\
\end{tabular}
\end{center}
\caption{
\label{table.drackns}
{\small 
Regular abelian \textsc{drackn} parameters $(n,r,c)$ that meet the necessary conditions with $n\leq 500$ and $r$ an odd prime.
Any such regular abelian \textsc{drackn} would necessarily produce equiangular tight frames of $n$ vectors in $\mathbb{C}^d$.
Next, $\delta=n-rc-2$ is a parameter of interest defined in~\cite{GodsilH:92}.
If such a regular abelian \textsc{drackn} is known by the authors to exist, its construction can be found in the reference(s) cited in the column labeled ``Existence.''
}\normalsize}
\end{table}

\subsection{Consequences for maximal equiangular tight frames}

Gerzon's bound implies that an equiangular tight frame in $\mathbb{C}^d$ necessarily has $n\leq d^2$ vectors~\cite{LemmensS:73}.
For this reason, ETFs that saturate this bound are known as \textbf{maximal} ETFs in the frame theory community.
It turns out that maximal ETFs find applications in quantum information theory, where they are known as symmetric, informationally complete positive operator--valued measures~\cite{FuchsS:11}.
Interestingly, Hadamard powers appear naturally in the context of maximal ETFs:

\begin{proposition}[Corollary~19 in~\cite{Waldron:17}]
\label{prop.cousin}
Given a maximal equiangular tight frame $\{\varphi_i\}_{i\in[d^2]}$ for $\mathbb{C}^d$ with signature matrix $\mathcal{S}$, then $\{\varphi_i^{\otimes 2}\}_{i\in[d^2]}$ forms an equiangular tight frame for its $\binom{d+1}{2}$-dimensional span with signature matrix $\mathcal{S}^{\circ 2}$.
\end{proposition}

It is widely believed that maximal ETFs exist in every dimension~\cite{Flammia:online,FuchsHS:17}.
This is the subject of Zauner's conjecture~\cite{Zauner:99}.
The following result establishes the extent to which maximal ETFs arise as \textsc{drackn} lines.

\begin{corollary}
\label{cor.no sics}
There do not exist $d^2$ \textsc{drackn} lines spanning $\mathbb{C}^d$.
There exist $d^2$ \textsc{drackn} lines spanning $\mathbb{C}^{d^2-d}$ only if $d=3$.
\end{corollary}

The second part above is Corollary~6.7 in~\cite{CoutinhoCSZ:16}, whereas the first part answers an open problem posed at the end of Section~6 in the same paper.
In particular, our result implies that none of the regular abelian \textsc{drackn}s satisfying case (II.a) of Theorem~6.5 in~\cite{CoutinhoCSZ:16} exist.
Our proof of both parts of Corollary~\ref{cor.no sics} uses the same technique, namely, Corollary~\ref{cor.drackn lines detector}.

\begin{proof}[Proof of Corollary~\ref{cor.no sics}]
Given a maximal ETF, then the Gram matrix is $\mathcal{G}=I+(1/\sqrt{d+1})\mathcal{S}$ by equality in the Welch bound~\eqref{eq.welch bound}.
Furthermore, the eigenvalues of $\mathcal{G}$ are $0$ and $d$, and so the eigenvalues of $\mathcal{S}$ are given by
\[
\sigma(\mathcal{S})
=\Big\{-\sqrt{d+1},~(d-1)\sqrt{d+1}\Big\}.
\]
By Proposition~\ref{prop.cousin}, the eigenvalues of $\mathcal{G}^{\circ 2}=I+(1/(d+1))\mathcal{S}^{\circ 2}$ are $0$ and $2d/(d+1)$, and so
\[
\sigma(\mathcal{S}^{\circ 2})
=\Big\{-(d+1),~d-1\Big\}.
\]
We claim that $\mathcal{S}$ and $-\mathcal{S}$ (and therefore their normalized versions) fail to satisfy Corollary~\ref{cor.drackn lines detector}(b) with one exception.
Indeed, the positive eigenvalues of $\mathcal{S}$ and $\mathcal{S}^{\circ 2}$ are equal only if $d=0$, whereas the positive eigenvalues of $-\mathcal{S}$ and $(-\mathcal{S})^{\circ 2}=\mathcal{S}^{\circ 2}$ are equal only if $d\in\{0,3\}$.
\end{proof}

Overall, \textsc{drackn} lines are too restrictive to produce maximal ETFs beyond $d=3$.
However, roux lines appear to be a fruitful relaxation in this regard.
For example, Corollary~\ref{cor.2tran-roux-etf}(a) gives that all three of the doubly transitive maximal ETFs classified in~\cite{Zhu:15} (namely, those in $\mathbb{C}^2$, the Hesse ETF in $\mathbb{C}^3$~\cite{Zauner:99}, and Hoggar's lines in $\mathbb{C}^8$~\cite{Hoggar:98}) span roux lines.
The following result provides another indication that roux lines may interact nicely with maximal ETFs:

\begin{corollary}
\label{cor.max ETF C4 roux}
Every maximal equiangular tight frame whose signature matrix consists of $4$th roots of unity is roux.
\end{corollary}

\begin{proof}
We will check (a) and (b) in Corollary~\ref{cor.roux lines detector}.
Since the signature matrix already consists of roots of unity, its normalized version will as well, and so we have (a).
For (b), note that the second Hadamard power has two eigenvalues by Proposition~\ref{prop.cousin}.
Also, the third Hadamard power is equivalent to applying the complex conjugate entrywise, which fixes the (real) minimal polynomial of the signature matrix.
\end{proof}

Explicit examples are given by~\eqref{eq.2by4 signature matrix}, and by Example~\ref{ex.maximal} below.

\subsection{Examples of doubly transitive lines}

We now describe every construction of doubly transitive lines that we currently know.

\begin{example}[Doubly transitive two-graphs]
Every two-graph (\cite{Seidel:76}) $\mathcal{T}$ on vertex set $[n]$ is known to create a sequence of equiangular lines $\mathscr{L} = \{ \ell_i \}_{i\in [n]}$ spanning $\mathbb{R}^d$ such that $\operatorname{Aut} \mathcal{T} = \operatorname{Aut} \mathscr{L}$, in the sense that every $\sigma \in S_n$ preserving the triple set $\mathcal{T}$ corresponds with an orthogonal matrix $U \in \operatorname{O}(d)$ satisfying $U \ell_i = \ell_{\sigma(i)}$ for every $i \in [n]$.
Considering the inclusions $\mathbb{R}^d \subset \mathbb{C}^d$ and $\operatorname{O}(d) \leq \operatorname{U}(d)$, every doubly transitive two-graph produces a sequence of doubly transitive lines spanning $\mathbb{C}^d$.
The doubly transitive two-graphs are all known~\cite{Taylor:92}, and their constructions are summarized in~\cite[\S 6]{Taylor:77}.
See~\cite[\S 2]{IversonM:future} for parameters of the resulting doubly transitive lines.
\end{example}

\begin{example}[Lines with $\operatorname{PSL}(2,q)$ symmetry]
Fix an odd prime power $q$, and write $\mathbb{F}_q$ for the finite field of order $q$, and $\mathbb{F}_q^\times$ for its group of units.
Let $\beta \colon \mathbb{F}_q^\times \to C_2$ be the linear character whose kernel consists of quadratic residues.
Put $G = {\operatorname{SL}(2,q) \times C_4}$ and
\[
H =
\left\{
\left( \begin{bmatrix} a & b \\ 0 & a^{-1} \end{bmatrix}, \beta(a) \right)
:
a \in \mathbb{F}_q^\times, b \in \mathbb{F}_q
\right\}
\leq
G.
\]
In the sequel to this paper~\cite{IversonM:future}, we prove that $(G,H)$ is a Higman pair with $N_G(H)/H \cong C_4$ and $n = [G : N_G(H)] = q+1$.
For an appropriate choice of key, Lemma~\ref{lem.higman's roux} creates a roux $B \in \mathbb{C}[C_4]^{n\times n}$ having parameters $c_1 = c_{-1} = (q-1)/2$ and $c_{\mathrm{i}} = c_{-\mathrm{i}} = 0$ when $q \equiv 1 \pmod 4$, while $c_1 = c_{-1} = 0$ and $c_{\mathrm{i}} = c_{-\mathrm{i}} = (q-1)/2$ when $q \equiv 3 \pmod 4$.
In either case, when we take $\alpha \in \widehat{C_4}$ to be the identity character we obtain the signature matrix $\mathcal{S} = \hat{\alpha}(B)$ of $q+1$ doubly transitive lines spanning $\mathbb{C}^{(q+1)/2}$.

Corollary~\ref{cor.real line detector} implies that $\mathcal{S}$ is the signature matrix of real lines if and only if $q \equiv 1 \pmod 4$.
When $q \equiv 3 \pmod 4$, $\mathcal{S}^{\circ 2} = \widehat{\alpha^2}(B)$ is the signature matrix of lines spanning $\mathbb{C}^q$, by Theorem~\ref{thm.roux primitive idempotents} and Lemma~\ref{lem.two reps of same lines}.
By Corollary~\ref{cor.drackn lines detector} we conclude that $\mathcal{S}$ is not the signature matrix of \textsc{drackn} lines whenever $q \equiv 3 \pmod 4$.
This gives an infinite family of roux lines that are not \textsc{drackn} lines.
\end{example}

\begin{example}[Lines with $\operatorname{PSU}(3,q)$ symmetry]
Fix a prime power $q > 2$.
We take $\operatorname{SU}(3,q)$ to consist of all matrices in $\operatorname{SL}(3,q^2)$ that stabilize the Hermitian form $(u,v) = u_1 v_3^q + u_2 v_2^q + u_3v_1^q$ on $\mathbb{F}_{q^2}^3$.
Let $\beta \colon \mathbb{F}_{q^2}^\times \to \mathbb{T}$ be any choice of nontrivial linear character such that $\operatorname{im} \beta =: C_{r} \leq C_{q+1}$.
Put $G = {\operatorname{SU}(3,q) \times C_{2r}}$ and
\[
H =
\left\{
\left(
\begin{bmatrix}
e & ea & eb \\
0 & e^{q-1} & - e^{q-1} a^q \\
0 & 0 & e^{-q}
\end{bmatrix},
\beta(e)
\right)
: e \in \mathbb{F}_{q^2}^\times, \,a,b \in \mathbb{F}_{q^2},\, a^{q+1} + b + b^q = 0
\right\}
\leq G.
\]
In the sequel paper~\cite{IversonM:future}, we prove that $(G,H)$ is a Higman pair with $n = [G : N_G(H)] = q^3 + 1$ and $N_G(H) / H \cong C_{2r}$.
After selecting an appropriate key, one can apply Lemma~\ref{lem.higman's roux} to obtain a roux $B \in \mathbb{C}[C_{2r}]^{n\times n}$ having parameters $\{ c_g \}_{g \in C_{2r}}$ given by
\[ c_g = \begin{cases}
\frac{q+1}{r}(q^2-1) + q - q^2, & \text{if }g=1; \\
\frac{q+1}{r}(q^2 - 1), & \text{if } g\in C_{r} \setminus \{1\}; \\
0, & \text{otherwise}.
\end{cases} \]
By Lemma~\ref{lem.inverse prob}, $B$ is switching equivalent to a roux $\tilde{B} \in \mathbb{C}[C_r]^{n\times n}$ having parameters
\[
c_g = \begin{cases}
\frac{q+1}{r}(q^2-1) + q - q^2, & \text{if }g=1; \\
\frac{q+1}{r}(q^2 - 1), & \text{if } g\in C_{r} \setminus \{1\}.
\end{cases}
\]
Theorem~\ref{thm.drackn vs roux} gives that $\lceil \tilde{B} \rfloor$ is the adjacency of a regular abelian $(q^3+1,r,(q-1)(q+1)^2/r)$-\textsc{drackn}.
(See~\cite{FickusJMPW:17} for regular abelian \textsc{drackn}s with these parameters.)
When $\alpha \in \widehat{C_r}$ is any nontrivial character, $\hat{\alpha}(\tilde{B})$ is the signature matrix of $q^3+1$ doubly transitive \textsc{drackn} lines spanning $\mathbb{C}^{q^2-q+1}$.
\end{example}

\begin{example}[Thas--Somma \textsc{drackn}s]
\label{ex.thas somma drackn}
The following example has been rediscovered several times in different forms~\cite{Thas:77,Somma:83,Higman:90,Cameron:91,GodsilH:92,IversonJM:16,BodmannKing:18}.
Fix a prime power $q$ and an integer $m \geq 1$.
Endow $V = \mathbb{F}_q^{2m}$ with a nondegenerate alternating bilinear form $[\cdot,\cdot] \colon V \times V \to \mathbb{F}_q$, and let $B \in \mathbb{C}[\mathbb{F}_q]^{V \times V}$ be given by $B_{uu} = 0$ and $B_{uv} = \delta_{[u,v]}$ for every $u \neq v \in V$, where $\delta_a \in \mathbb{C}[\mathbb{F}_q]$ denotes the basis vector corresponding to $a \in \mathbb{F}_q$.
Then $B$ is easily seen to satisfy (R1)--(R3), while a straightforward computation gives $B^2 = (q^{2m}-1)I + \sum_{a \in \mathbb{F}_q} c_a \delta_a B$, where
\[
c_a = \begin{cases}
q^{2m-1} - 2, & \text{if }a = 0; \\
q^{2m-1}, & \text{otherwise.}
\end{cases}
\]
By Lemma~\ref{lem.B squared}, $B$ is a roux for $\mathbb{F}_q$ with parameters $\{c_a \}_{a \in \mathbb{F}_q}$ given above.
Indeed, Theorem~\ref{thm.drackn vs roux} implies that $\lceil B \rfloor$ is the adjacency matrix of a regular abelian $(q^{2m},q,q^{2m-1})$-\textsc{drackn}, which is sometimes called the Thas--Somma construction.
For any nontrivial $\alpha \in \widehat{\mathbb{F}_q}$, $\mathcal{S} = \hat{\alpha}(B)$ is the signature matrix of a sequence $\mathscr{L}$ of $q^{2m}$ lines spanning $\mathbb{C}^{q(q^m+1)/2}$, by Theorem~\ref{thm.roux primitive idempotents} and Lemma~\ref{lem.two reps of same lines}.
When $q$ is an odd prime, alternative constructions of $\mathscr{L}$ appear as special cases of~\cite[Theorem~6.4]{IversonJM:16} and~\cite[Theorem~4.10]{BodmannKing:18}.
(While the latter constructions are not obviously the same as the one above, one can easily prove their equivalence by considering signature matrices.)
Notably, $\mathscr{L}$ can be chosen as a projective orbit for a unitary representation of a Heisenberg group whenever $q$ is an odd prime.

We now explain how $\mathscr{L}$ is doubly transitive.
Denote $\operatorname{Sp}(2m,q) \leq \operatorname{GL}(2m,q)$ for the group of matrices that stabilize the form $[ \cdot, \cdot ]$.
The group $G = V \rtimes \operatorname{Sp}(2m,q)$ permutes $V$ with the affine action $(u,M)\cdot v = Mv + u$.
For any ring $R$, $G$ acts on $R^{V \times V}$ from the left by permuting indices:
\[
[(u,M)\cdot A]_{v,w} = A_{(u,M)^{-1}\cdot v, (u,M)^{-1}\cdot w} = A_{M^{-1}(v-u), M^{-1}(w-u)}.
\]
Given $u \in V$, let us denote $D_u \in \mathbb{C}[\mathbb{F}_q]^{V\times V}$ for the diagonal matrix with $[D_u]_{vv} = \delta_{[v,u]}$ for every $v \in V$.
We also abbreviate $\hat{D}_u = \hat{\alpha}(D_u) \in \mathbb{C}^{V \times V}$.
Then one easily checks that $(u,M)\cdot B = D_u^{-1} B D_u$ for every $u \in V$ and $M \in \operatorname{Sp}(2m,q)$.
Consequently,
\[
(u,M) \cdot \mathcal{S} = \hat{\alpha}\bigl( (u,M)\cdot B \bigr) = \hat{D}_u^* \mathcal{S} \hat{D}_u.
\]
From this it follows that $G \leq \operatorname{Aut} \mathscr{L}$ (cf.\ Lemma~2.3 in \cite{ChienW:18}).
Since $G$ is known to act doubly transitively on $V$~\cite{Grove:02}, we conclude that $\mathscr{L}$ is doubly transitive.
\end{example}

\begin{example}[Hoggar's lines]
\label{ex.maximal}
Take $h\in L^2(\mathbb{Z}_2^3)$ defined by $h(0)=-1+2\mathrm{i}$ and $h(j)=1$ for $j\neq 0$ (as given in~\cite{JedwabW:15,SzymusiakS:16}), and let $T^a$ and $M^b$ denote translation and modulation operators over $L^2(\mathbb{Z}_2^3)$:
\[
(T^af)(x)=f(x+a),
\qquad
(M^bf)(x)=(-1)^{b\cdot x}f(x).
\]
Then $\{T^aM^bh\}_{a,b\in\mathbb{Z}_2^3}$ is a maximal ETF (namely, Hoggar's lines) in which every off-diagonal entry of its signature matrix $\mathcal{S}_\mathrm{H}$ lies in $C_4$.
By Corollary~\ref{cor.max ETF C4 roux}, the vectors in this ETF span roux lines.
To see this, for each $k$, consider the matrix $B_k\in\mathbb{C}[C_4]^{2^{2k}\times 2^{2k}}$ with indices in $(\mathbb{Z}_2^k)^2$ defined by
\[
(B_k)_{(a,b),(c,d)}
:=\left\{
\begin{array}{cl}
0&\text{if }(a,b)=(c,d);\\
\phantom{_-}\delta_{\mathrm{i}^{\textsc{gray}^{-1}(d\cdot(a+c),b\cdot(a+c))}}&\text{else if }a=c\text{ or }b=d;\\
\delta_{-\mathrm{i}^{\textsc{gray}^{-1}(d\cdot(a+c),b\cdot(a+c))}}&\text{otherwise}.
\end{array}
\right.
\]
(As usual, $\delta_g$ denotes the image of $g\in C_4$ in $\mathbb{C}[C_4]$.)
Here, $\textsc{gray}\colon\mathbb{Z}_4\to\mathbb{Z}_2^2$ maps $j\in\mathbb{Z}_4$ to the $j$th Gray codeword~\cite{Gray:47,CalderbankCKS:97}, i.e.,
\[
\textsc{gray}\colon
\quad
0\mapsto(0,0);
\quad
1\mapsto(0,1);
\quad
2\mapsto(1,1);
\quad
3\mapsto(1,0).
\]
Then $B_k$ is a roux for $C_4$ when $k\in\{1,3\}$.
In particular, evaluating $B_1$ at the character $\alpha$ defined by $\alpha(z)=z$ gives the signature matrix of a maximal ETF in $\mathbb{C}^2$, whereas doing the same for $B_3$ produces the signature matrix $\mathcal{S}_\mathrm{H}$ of Hoggar's lines.
The roux parameters of $B_1$ are $c_1=c_{-1}=0$ and $c_\mathrm{i}=c_{-\mathrm{i}}=1$ (cf.\ Lemma~\ref{lem.conference roux}), whereas the roux parameters of $B_3$ are
\[
c_1=24,
\qquad
c_\mathrm{i}=c_{-\mathrm{i}}=16,
\qquad
c_{-1}=6.
\]

Hoggar's lines are known to extend to an infinite family of maximal ETFs over a finite field~\cite{GreavesIJM:20}.
Interestingly, the roux parameters of $B_3$ generalize to an infinite family, leaving open the possibility of an infinite family of maximal complex ETFs that arise from roux lines.
In particular, for any positive integer $j$, consider the parameters
\[
c_1
=4j^4+12j^3+10j^2-2,
\qquad
c_\mathrm{i}=c_{-\mathrm{i}}
=4j^4+8j^3+4j^2,
\qquad
c_{-1}
=4j^4+4j^3-2j^2.
\]
($B_3$ exhibits these parameters with $j=1$.)
Summing these parameters (and adding $2$) gives $n=16j^2(j+1)^2$.
Take $\alpha\in\widehat{C_4}$ defined by $\alpha(z)=z$.
Then $d:=d_\alpha^+=4j(j+1)=\sqrt{n}$, meaning that for any roux with these parameters, evaluating at $\alpha$ produces the signature matrix of a maximal ETF in $\mathbb{C}^d$.
Furthermore, the integrality condition in Corollary~\ref{cor.roux lines integrality} is satisfied with $\sqrt{q}\in\mathbb{Z}$.
In addition, $d_{\alpha^2}^+=\binom{d+1}{2}$, matching the necessary condition in Proposition~\ref{prop.cousin}.
Finally, a real ETF of this size exists if and only if there exists a regular symmetric Hadamard matrix of constant diagonal, and such matrices necessarily exist whenever there is a Hadamard matrix of order $d=4j(j+1)$~\cite{FickusM:online}; the Hadamard conjecture implies that such a matrix exists for every $j$.
\end{example}

\section{Summary and remaining proofs}
\label{sec.summary}

As demonstrated by Corollary~\ref{cor.2tran-roux-etf}, roux arise naturally in the study of doubly transitive lines.
Indeed, the technology developed here plays a prominent role in the sequel paper~\cite{IversonM:future}, which classifies all doubly transitive lines having almost simple symmetries.
Roux simultaneously generalize doubly transitive lines, regular abelian \textsc{drackn}s, and regular two-graphs.
Mathematically, the study of roux lies in the intersection of group theory (roux proper), algebraic combinatorics (roux schemes), discrete geometry (roux lines), and graph theory (roux graphs); the theory favors a rich interplay between these perspectives. 
They have applications for equiangular lines and regular abelian \textsc{drackn}s, however we believe they are worthy of study in their own right.
The next steps are to find more constructions.
Example~\ref{ex.maximal} indicates one promising direction which, if fruitful, would give a combinatorial approach to Zauner's conjecture~\cite{Zauner:99}.
In addition, Table~\ref{table.drackns} points out many open problems regarding the existence of roux.

What follow are the proofs of the results reported in Subsection~\ref{subsec.example main results}.
The proof of Theorem~\ref{thm.Higman Pair Theorem}(a) appears in the sequel paper~\cite{IversonM:future}.

\begin{proof}[Proof of Theorem~\ref{thm.Higman Pair Theorem}(b)]
Given a Higman pair $(G,H)$, then the Schurian scheme of $(G,H)$ is isomorphic to a roux scheme by Lemma~\ref{lem.higman's roux}.
In particular, this scheme arises from an $n\times n$ roux for $\Gamma=K/H$, where $K=N_G(H)$.
Since the roux scheme's adjacency algebra is necessarily commutative, we may conclude that $(G,H)$ is a Gelfand pair.
Next, Theorem~\ref{thm.roux primitive idempotents} provides the primitive idempotents of the roux scheme, each of which is the Gram matrix of $r$ equal-norm representatives from each of $n$ lines, where the rank is strictly smaller than $n$ and the phase of each entry is an $r$th root of unity.
It remains to show that the automorphism group of the lines contains the action of $G$ on $G/K$.

To this end, fix some $\mathcal{G}=\mathcal{G}_\alpha^\epsilon$ with rank $d=d_\alpha^\epsilon$ from Theorem~\ref{thm.roux primitive idempotents}.
Since $\mathcal{G}$ lies in $\mathscr{A}(G,H)$, there exists a unitary representation $\pi\colon G\to \operatorname{U}(d)$ and a vector $v\in\mathbb{C}^d$ such that $\pi(h)v=v$ for every $h\in H$ and $\{\pi(x_ja_g)v\}_{j\in[n],g\in K/H}$ has Gram matrix $\mathcal{G}$ by Theorem~3.2 in~\cite{IversonJM:17} (here, we follow Lemma~\ref{lem.higman's roux} in selecting left coset representatives $\{x_j\}_{j\in[n]}$ for $K$ in $G$ and coset representatives $\{a_g\}_{g\in K/H}$ for $H$ in $K$).
By Theorem~\ref{thm.roux primitive idempotents}, the entries of $\mathcal{G}$ that have modulus $1$ appear in the $r\times r$ diagonal blocks; these are the entries $\mathcal{G}_{(i,g),(i,h)}$ for $i\in[n]$ and $g,h\in\Gamma=K/H$.
As such, we may conclude that $|\langle \pi(x_ia_g)v,\pi(x_ja_h)v\rangle|=1$ if and only if $i=j$, which in turn means $|\langle \pi(x)v,\pi(y)v\rangle|=1$ for $x,y\in G$ if and only if $x^{-1}y\in K$.

Now we consider how $G$ acts on the lines spanned by $\{\pi(x_ja_g)v\}_{j\in[n],g\in K/H}$ under the action $g\cdot[\pi(x)v]=[\pi(g)\pi(x)v]$.
This action is transitive, and furthermore, since $g\cdot[v]=[v]$ if and only if $|\langle v,\pi(g)v\rangle|=1$, if and only if $g\in K$, we have that $K$ is the stabilizer of $[v]$ under this action.
By the orbit--stabilizer theorem, this action is equivalent to the action of $G$ on $G/K$, which is doubly transitive by (H1).
\end{proof}

\begin{proof}[Proof of Corollary~\ref{cor.2tran-roux-etf}]
Part (a) of the Higman Pair Theorem (see~\cite{IversonM:future}) gives that doubly transitive lines arise from primitive idempotents of the corresponding Schurian scheme, which is a roux scheme by Theorem~\ref{thm.schurian roux}, and so the primitive idempotents produce roux lines by Lemma~\ref{lem.two reps of same lines}.
This gives (a), while (b) follows immediately from Theorem~\ref{thm.roux signature}.
\end{proof}

\section*{Acknowledgments}

The authors thank the anonymous referees for comments that substantially improved the quality of the paper.
Furthermore, one of the referees generously provided a preliminary version of Figure~\ref{figure.diagram}.
This work was partially supported by NSF DMS 1321779, NSF DMS 1829955, ARO W911NF-16-1-0008, AFOSR F4FGA06060J007, AFOSR FA9550-18-1-0107, AFOSR Young Investigator Research Program award F4FGA06088J001, an AFRL Summer Faculty Fellowship Program award, and the Simons Institute of the Theory of Computing
The authors thank Matt Fickus for ongoing, stimulating conversations, and in particular, for assisting in the preparation of Table~\ref{table.drackns}.
JWI thanks the Norbert Wiener Center and the Department of Mathematics, University of Maryland, College Park for their hospitality.
The views expressed in this article are those of the authors and do not reflect the official policy or position of the United States Air Force, Army, Department of Defense, the U.S.\ Government, Graham Higman, or Donald G.~Higman.

\end{document}